\documentclass{article}
\usepackage[utf8]{inputenc}
\usepackage[margin = 1in]{geometry}

\usepackage{amsmath}
\usepackage{hyperref}
\usepackage{amsthm}

\newtheorem{proposition}{Proposition}
\theoremstyle{definition}
\newtheorem{definition}{Definition}

\theoremstyle{remark}
\newtheorem{remark}{Remark}
\newtheorem{corollary}{Corollary}

\usepackage[T1]{fontenc}
\usepackage[table]{xcolor}

\usepackage{ dsfont }
\def\bm{\ensuremath\mathbf}
\usepackage{tcolorbox}

\usepackage{lipsum}
\usepackage{amsfonts}
\usepackage{graphicx}
\usepackage{epstopdf}

\usepackage{algorithm}
\usepackage{algorithmic}

\usepackage{enumitem}
\setlist[enumerate]{leftmargin=.5in}
\setlist[itemize]{leftmargin=.5in}

\title{Analytical solutions for some quadratic ODEs found via linear rational eigenfunctions and
the rational eigenfunction variety}

\author{Megan Morrison$^{1*}$ and Sonja Petrović$^1$\\[.1in]
$^1$ Department of Applied Mathematics, Illinois Institute of Technology, Chicago, IL 60616\\
$^*$ mmorrison3@illinoistech.edu
}

\begin{document}

\maketitle

\begin{abstract}
Many important systems across biology, engineering, physics, and economics are characterized by polynomial ordinary differential equations (ODEs), yet analytical solutions are rare.  
We develop a framework for identifying and solving a broad class of two-dimensional quadratic ODEs using linear rational Koopman eigenfunctions.  By imposing a linear rational form on the eigenfunctions, we convert the Koopman eigenfunction PDE into a large algebraic system of polynomials.
We then study the solutions of this polynomial system that satisfy the ODE restrictions; we call the solution set the \emph{rational eigenfunction variety} of an ODE system. 
The nonlinear algebra method uses formal algebraic geometry theory to analyze and solve systems otherwise intractable and to discover relationships between ODE and eigenfunction parameters that must hold to extract eigenfunctions.
We identify families of quadratic ODEs that can be solved analytically, characterize their eigenfunction parameters, and use the resulting eigenfunctions to produce closed-form analytical solutions. \\

\textit{Keywords:}
polynomial ODEs, Koopman eigenfunctions, elimination ideal, singular affine varieties. 
\end{abstract}


\section{Introduction}

\begin{figure}[th!]
    \centering
    \includegraphics[width=0.95\linewidth]{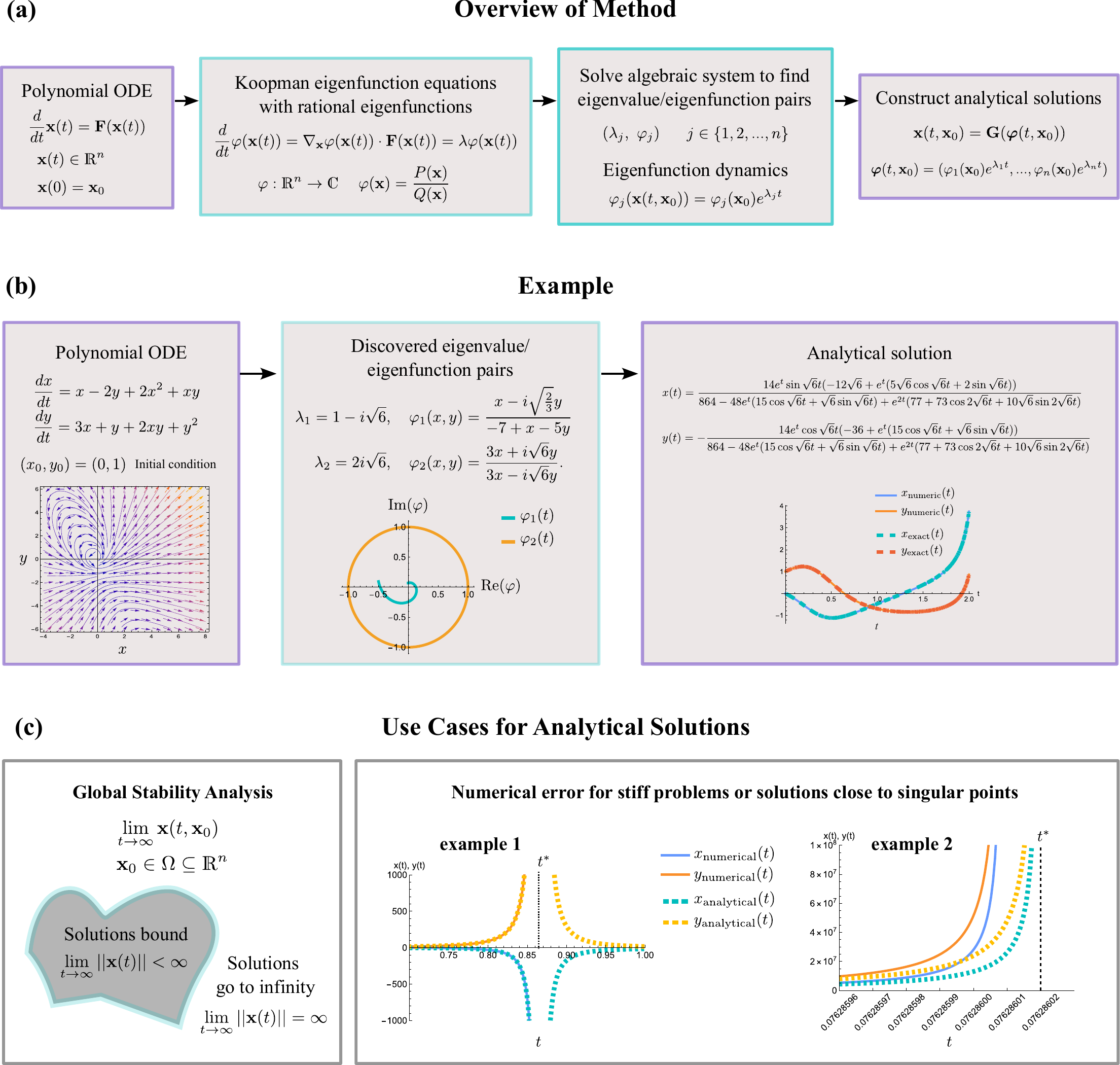}
    \caption{(a) Rational eigenfunction method to solve polynomial ODEs. (b) Example ODE and solution from method. (c) Use cases for analytical solutions --- stability analysis and accurate solutions for stiff problems.}
    \label{fig:intro}
\end{figure}

In contrast to linear systems, nonlinear systems of ordinary differential equations are typically solved numerically, as methods to solve them analytically generally do not exist \cite{kutzDataDrivenModelingScientific2013, vidyasagarNonlinearSystemsAnalysis2002}.
Advanced numerical integration methods such as Runge–Kutta and Adams–Bashforth methods provide a high degree of accuracy. Still, analytical solutions are desirable as they are amenable to analysis, more accurate, and less computationally expensive than numerical solutions \cite{strogatzNonlinearDynamicsChaos2019, braunDifferentialEquationsTheir1978}.
When exact solutions do not exist, approximate analytical solutions are still useful for constructing control methods such as model predictive control and feed-forward control \cite{findeisenStateOutputFeedback2003, morrisonNonlinearControlNetworked2021}.

\smallskip

\paragraph{Koopman methods.}
Koopman methods are a promising avenue for solving nonlinear dynamical systems.  
Koopman theory states that any nonlinear ODE has linear dynamics in the space of observables; the trade-off is that these linear dynamics occur in an infinite-dimensional space \cite{koopmanHamiltonianSystemsTransformation1931, koopmanDynamicalSystemsContinuous1932, bruntonModernKoopmanTheory2021}.
Eigenfunctions of the Koopman operator are special observables that have linear dynamics, and if they are able to be identified, can be used to construct solutions for the original system \cite{bruntonKoopmanInvariantSubspaces2016, morrisonSolvingNonlinearOrdinary2024}. 
Unfortunately, there are no general methods to find eigenfunctions that have a closed-form for systems of ODEs; closed-form expressions for Koopman eigenfunctions are only discoverable in select cases \cite{bolltMatchingEvenRectifying2018,kaiserDatadrivenDiscoveryKoopman2021,bolltGeometricConsiderationsGood2021,pageKoopmanModeExpansions2019}.
While eigenfunction approximations are useful for solution approximations, control problems, and analysis, closed-form eigenfunctions are needed to obtain analytical solutions to ODEs \cite{kordaOptimalConstructionKoopman2020, budisicAppliedKoopmanism2012, kutzDynamicModeDecomposition2016, bruntonDataDrivenScienceEngineering2019}.

\smallskip

\textit{Using methods from nonlinear algebra, we identify closed-form eigenfunctions for classes of quadratic ODEs that currently lack analytical solutions.}

\smallskip

\paragraph{Nonlinear algebra.}
Nonlinear algebra -- a term coined relatively recently that refers to the computationally-inspired view of algebraic geometry -- is a study of solution sets of systems of polynomial equations. The key here is that the methods we use are symbolic, exact, and designed to identify solutions to high-dimensional nonlinear systems, answering questions such as: are there common solutions to a system of nonlinear restrictions that a given ODE has to satisfy? How many solutions are there, and what is the dimension of the solution space? 
Polynomial algebraic systems show up in many applications; we will show how they appear in particular when solving high-dimensional systems generated by nonlinear ODEs and structural restrictions on their eigenfunctions. For a friendly overview of the field, the interested reader may consult the textbook \cite{CLO}, and a more recent book \cite{NLA-Bernd} and overview paper \cite{BreidingEtAl2023}.

In the proposed method (overview shown in Fig.~\ref{fig:intro}(a)) we: 
\begin{enumerate}[noitemsep]
    \item Posit a quadratic polynomial ODE with undetermined parameters;
    \item Construct Koopman eigenfunction equations using parameterized linear rational eigenfunctions;
    \item Extract algebraic equations for the parameters from the eigenfunction equations and solve them using nonlinear algebra;
    \item Use the derived parameters to construct analytical solutions for the ODE.
\end{enumerate}
%
Fig.~\ref{fig:intro}(b) shows an example 2-dimensional quadratic ODE, the derived eigenfunctions using our method, and the resulting analytical solution to the initial value problem posed.

\smallskip

\paragraph{Applications.}
Analytical solutions to systems of ordinary differential equations have many advantages and use cases.  They can be used to determine the stability of a system throughout the entire state space and construct novel control procedures.
Analytical solutions are also advantageous for stiff ODEs where numerical error can be quite large and prohibitively expensive to reduce.  In Fig.~\ref{fig:intro}(c) example 1, the true (analytical) solution spikes around $t^*$ before returning close to zero.  For this initial value problem the numerical solution fails past $t^*$ as the problem is too stiff to solve despite using an adaptive, stiffness-aware high order RK method, implemented by the numerical differential equation solver NSolve in \textit{Mathematica}. In Fig.~\ref{fig:intro}(c) example 2, using the same solver, the numerical solution diverges substantially from the analytical solution as the solution approaches the blow-up time $t^*$.
Analytical solutions enable novel approaches for the analysis and control of ODEs and provide accurate solutions when numerical methods fail.

\smallskip

\paragraph{Paper organization.}
The paper is organized as follows. Section~\ref{sec:eigfunc} describes Koopman eigenfunctions, quadratic ODEs, the eigenfunction equations, and the polynomial system of algebraic equations that arise from imposing a linear rational form on the eigenfunctions. Section~\ref{sec:solve_ode} outlines a method to solve nonlinear ODEs using eigenfunctions. Section~\ref{sec:alg_geo} outlines the high-level of the nonlinear algebra methods used to solve the system of equations generated by the eigenfunctions, with Appendix~\ref{sec:alg_geo_appx} providing further details. 
Section~\ref{sec:solvable_ODEs_eig_params} describes the set of solvable quadratic ODEs and their resulting eigenfunctions.
This is the bridge between the differential equation world and the algebraic world. It tells us exactly which quadratic ODEs are solvable using this method, and how to find the eigenfunctions needed to build the analytical solutions.
 Section~\ref{sec:examples} contains examples of solvable quadratic ODEs and their solutions.
 Section~\ref{sec:discussion} compares our method to other Koopman methods and discusses potential extensions of this approach.
 Section~\ref{sec:conclusions} concludes the manuscript.

\section{Koopman approach to ODEs and conversion to algebraic systems}\label{sec:eigfunc}

We will take a Koopman approach to solving polynomial ODEs by imposing a rational form on the eigenfunctions, extracting algebraic equations for the parameters, and then solving the algebraic equations to determine the eigenfunction parameters.

\subsection{Background: Koopman theory}

\paragraph*{The Koopman operator.}
Koopman operator theory states that any system of nonlinear ODEs can be mapped to a space of observables with linear dynamics.  
Consider the ODE
$$\frac{d}{dt}\bm{x}(t) = \bm{F}(\bm{x}(t)), \ \ \bm{x} \in \mathds{R}^n,$$
where $\bm{F}: \mathds{R}^n \to \mathds{R}^n$ is an autonomous vector field that acts on the state vector $\bm{x}$.  Let $\bm{x}(t, \bm{x}_0)$ be the flow from initial condition $\bm{x}_0$.
For all $t$, the Koopman operator describes the dynamics of \textit{observables} of the state vector along its flow \cite{koopmanHamiltonianSystemsTransformation1931, budisicAppliedKoopmanism2012, mezicSpectralPropertiesDynamical2005}. Each observable $g: \mathds{R}^n \to \mathds{C}$ is an element of a space of observable functions $\mathcal{F}$.  The Koopman operator $\mathcal{K}_t: \mathcal{F} \to \mathcal{F}$ is an infinite-dimensional linear operator that propagates observables $g$ of the state vector $\bm{x}$ forward in time along trajectories of the ODE,
$$\mathcal{K}_t[g](\bm{x}_0) = g \circ \bm{x}(t, \bm{x}_0).$$
The left-hand side of this equation states that the Koopman operator $\mathcal{K}_t$ pushes the observables $g$ of the initial state vector $\bm{x}_0$ forward in time $t$.  
The right-hand side states that the initial state vector $\bm{x}_0$ is pushed forward in time $t$ according to the flow and then observed by $g$.  

\smallskip

\paragraph*{Eigenfunctions.}
Eigenfunctions $\varphi \in \mathcal{F}$ of the Koopman operator are special observables \cite{bolltGeometricConsiderationsGood2021}.  An eigenvalue-eigenfunction pair $(\lambda, \varphi)$ of $\mathcal{K}_t$ satisfies the equation
$$\mathcal{K}_t [\varphi](\bm{x}_0) = \varphi(\bm{x}_0) e^{\lambda t}.$$
Eigenfunctions have linear dynamics and satisfy the equation 
$$\frac{d}{dt}\varphi(\bm{x}(t)) = \lambda \varphi(\bm{x}(t)).$$
Eigenfunctions are highly useful yet challenging to acquire.
Data-driven methods such as \textit{extended dynamic mode decomposition} (EDMD) can extract Koopman eigenfunction approximations which are useful for both understanding the dynamics of the original system and for applications such as control \cite{williamsDataDrivenApproximation2015, kutzDynamicModeDecomposition2016, arbabiErgodicTheoryDynamic2017,williamsKernelbasedMethodDatadriven2015}. However, due to the form and approximate nature of these data-driven representations, eigenfunction approximations have not yet been helpful for obtaining closed-form analytical solutions to the original ODE.

For select systems, carefully chosen Koopman observables form an invariant subspace resulting in an exact, finite-dimensional linear representation of the nonlinear dynamics \cite{bruntonKoopmanInvariantSubspaces2016, haseliLearningKoopmanEigenfunctions2022}.  The Koopman eigenfunctions of these invariant subspaces have a closed-form.
If eigenfunctions have a closed-form, they can be used to construct analytical solutions for the original ODE given that the eigenfunctions form a large enough independent set.
In previous work we developed a technique to construct closed-form eigenfunctions for certain nonlinear ODEs using the system's invariant manifolds \cite{morrisonSolvingNonlinearOrdinary2024}. 
We now turn to computational algebraic geometry to find closed-form eigenfunctions for a broader class of ODEs.

We aim to identify all quadratic polynomial ODEs that can be solved using closed-form linear rational eigenfunctions.  After determining ODEs that possess qualities necessary to extract closed-form eigenfunctions, we  derive formulas for these eigenfunctions and subsequently construct analytical solutions to the ODEs.

\subsection{2D Quadratic ODEs and the eigenfunction equation}

Consider the system of ordinary differential equations $\frac{d}{dt}\bm{x}(t)= \bm{F}(\bm{x}(t))$ where $\bm{F}$ is a quadratic multivariate polynomial in variables $\bm{x}(t) \in \mathds{R}^2$. 
We will, in the interest of brevity, from here on suppress the dependence of $\bm{x}$ on $t$ in our notation and write $\frac{d}{dt}\bm{x}= \bm{F}(\bm{x})$.  In component form, the ODE system can be expressed as
\begin{align}
\begin{split}\label{eq:ODE_gen}
    \frac{dx}{dt} &= p_1(x,y) =  a_1 x + a_2 y + a_3 x^2 + a_4 xy + a_5 y^2\\
    \frac{dy}{dt} &= p_2(x,y) =  b_1 x + b_2 y + b_3 x^2 + b_4 xy + b_5 y^2
    \end{split}
\end{align}
with variables $(x, y)$ and real parameters, $\bm{a} = [a_1, a_2, ..., a_5] \in \mathds{R}^5$ and $\bm{b} = [b_1, b_2,..., b_5] \in \mathds{R}^5$.  We define this to be the normal form for quadratic ODEs --- without added constants $a_0$ and $b_0$.  Quadratic ODEs with added constants can be converted to this normal form using a change of variables (see Appendix~\ref{sec:normal_form}).

We ask the following question: what subset of quadratic ordinary differential equations, Eq.~\ref{eq:ODE_gen}, have linear rational eigenfunctions?  In other words, what relationship must exist among the parameters $\bm{a}$ and $\bm{b}$ to guarantee an eigenfunction of the form
\begin{align}\label{eq:eigenfunction}
    \varphi(x,y) &= \frac{c_0 + c_1 x + c_2 y}{d_0 + d_1 x + d_2 y}
\end{align}
where $\bm{c} = [c_0, c_1, c_2] \in \mathds{C}^3$ and $\bm{d} = [d_0, d_1, d_2] \in \mathds{C}^3$.
Two independent eigenfunctions are required to construct a solution to the ODE, Eq.~\ref{eq:ODE_gen}; we therefore wish to identify ODEs that possess at least two independent linear rational eigenfunctions \cite{morrisonSolvingNonlinearOrdinary2024}.  
Identification of such ODEs can be done by first imposing conditions on the coefficients so that a nontrivial solution exists, and then extracting those points that are singular, which would  correspond to at least two independent solutions. 
As we will see below, the conditions are naturally described as a system of polynomial equations, therefore these steps can be done by exact computations using  methods from computational algebraic geometry.

\smallskip

\paragraph*{Eigenfunction equation.}
An eigenfunction, $\varphi: \mathds R^n \to \mathds C$ of Eq.~\ref{eq:ODE_gen} is a function that satisfies
\begin{align*}
    \frac{d}{dt}\varphi(\bm{x}(t)) &= \lambda \varphi(\bm{x}(t)) 
\end{align*}
for some $\lambda \in \mathds{C}$.  
Expanding the left-hand side and using ODE (\ref{eq:ODE_gen}) produces
\begin{align}\label{eq:dyn_eig}
    \nabla_{\bm{x}} \varphi(\bm{x}) \cdot \frac{d\bm{x}}{dt} = \nabla_{\bm{x}} \varphi(\bm{x}) \cdot \bm{F}(\bm{x}) &=  \lambda \varphi(\bm{x}).
\end{align}
Restricting the ODE dynamics to a quadratic form (Eq.~\ref{eq:ODE_gen}), yields
\begin{align}\label{eq:dyn_wpoly}
\frac{\partial \varphi}{\partial x} (a_1 x + a_2 y + a_3 x^2 + a_4 xy + a_5 y^2) + \frac{\partial \varphi}{\partial y} (b_1 x + b_2 y + b_3 x^2 + b_4 xy + b_5 y^2)  &= \lambda \varphi
\end{align}

Without boundary or initial conditions, a large variety of possible solutions exists.  Given, however, that our goal is to use the eigenfunctions to construct analytical solutions for the original ODE, we wish to identify closed-form eigenfunction solutions $\varphi$ that are as simple as possible and amenable to the construction of ODE solutions.

\subsection{Eigenfunctions with a rational form}

In previous work \cite{morrisonSolvingNonlinearOrdinary2024}, we identified that certain polynomial ODEs have exact eigenfunctions of a rational form as is shown in the following example. 

\begin{tcolorbox}[colback=gray!5,colframe=blue!60!purple,title=Example: ODE with closed-form rational eigenfunctions]
Using the invariant manifold method introduced in Ref.\cite{morrisonSolvingNonlinearOrdinary2024}, we found that the nonlinear system of ODEs
\begin{equation}
\begin{split}\label{eq:intro_ode_ex}
    \dot{x} &= xy\\
    \dot{y} &= y^2 - x - 1
    \end{split}
\end{equation}
has eigenvalue/eigenfunction pairs
\begin{align*}
    \Big( \lambda_1 = 1, \quad \varphi_1(x,y) = \frac{x}{1+x+y} \Big ), \quad \text{and} \quad
    \Big(\lambda_2 = -1, \quad \varphi_2(x,y) = \frac{x}{1+x-y} \Big ).
\end{align*}
Using these exact eigenfunctions, we constructed exact analytical solutions to Eq.~\ref{eq:intro_ode_ex}:
\begin{align}
    \begin{split}
    x(t; x_0, y_0) &= \frac{2x_0 e^t}{1+x_0-2x_0 e^t + e^{2t}(1+x_0 - y_0) + y_0}\\
     y(t; x_0, y_0) &= \frac{1 + x_0 + y_0 + e^{2t}(-1 - x_0 + y_0)}{1+ x_0 - 2 x_0 e^t + e^{2t}(1 + x_0 - y_0) + y_0}
    \end{split}
\end{align}
where $(x_0, y_0)$ is the initial condition.
\end{tcolorbox} Two independent eigenfunctions (eigenfunctions that do not belong to the same equivalence class \cite{bolltGeometricConsiderationsGood2021, morrisonSolvingNonlinearOrdinary2024}) were required to obtain the analytical solution, $(x(t), y(t))$. 
The invariant manifold method takes closed-form representations of invariant manifolds in the system and uses them for the numerator and denominator of rational eigenfunction candidates \cite{morrisonSolvingNonlinearOrdinary2024}.  This method is effective at finding eigenfunctions only in select cases; it is highly restrictive and cannot be generally applied to the vast majority of systems of polynomial ODEs.

The resulting rational forms that emerge raise the question: What is the complete set of polynomial ordinary differential equations that have rational eigenfunctions?  Furthermore, can rational eigenfunctions be computed directly from the parameters that define an ODE? 
We answer these questions for the more restrictive set of \textit{quadratic} two-dimensional ODEs with \textit{linear} rational eigenfunctions.

\smallskip

\paragraph*{Restrict eigenfunctions to a linear rational form in the eigenfunction equation.}
We impose a rational form on the eigenfunctions,
\begin{align*}
    \varphi(\bm{x}) = \frac{P(\bm{x})}{Q(\bm{x})}
\end{align*}
where $P(\bm{x}), Q(\bm{x}) \in \mathcal{P}_1(\mathds{R}^2)$ are multivariate polynomials with an imposed finite degree $\leq 1$,
\begin{align*}
    P(x,y) &= c_0 + c_1 x + c_2 y, \\
    Q(x,y) &= d_0 + d_1 x + d_2 y.
\end{align*}
The eigenfunction coefficients can be complex-valued, $\bm{c} = [c_0, c_1, c_2] \in \mathds{C}^3$ and $\bm{d} = [d_0, d_1, d_2] \in \mathds{C}^3$.
The partial derivative of eigenfunction $\varphi(\bm{x})$ with respect to $x_i$ is
\begin{align*}
    \frac{\partial}{\partial x_i}\varphi(\bm{x}) = \frac{Q(\bm{x}) \frac{\partial}{\partial x_i}P(\bm{x}) - P(\bm{x}) \frac{\partial}{\partial x_i}Q(\bm{x})}{Q^2(\bm{x})}.
\end{align*}
With the imposed restrictions, the partial derivatives are
\begin{align*}
    \frac{\partial \varphi}{\partial x} &= \frac{-d_1(c_0+c_2 y)+c_1(d_0+d_2 y)}{(d_0+d_1 x+d_2 y)^2} \\
    \frac{\partial \varphi}{\partial y} &= \frac{-d_2(c_0+c_1 x)+c_2(d_0+d_1 x)}{(d_0+d_1 x+d_2 y)^2}.
\end{align*}

Inserting the rational form for the eigenfunction and its partial derivatives into Eq.~\ref{eq:dyn_wpoly} and moving all terms to the left-hand side produces 
\begin{align}\label{eq:dyn_wpoly_and_rational}
    \begin{split}
\Omega(x,y) :=& \lambda \varphi(x,y) Q^2(x,y)\\
&-  \left( Q(x,y) \frac{\partial}{\partial x}P(x,y) - P(x,y) \frac{\partial}{\partial x}Q(x,y) \right) p_1(x,y) \\
 &-  \left( Q(x,y) \frac{\partial}{\partial y}P(x,y) - P(x,y) \frac{\partial}{\partial y}Q(x,y) \right) p_2(x,y) 
=  0.
    \end{split}
\end{align}
$\Omega(x,y)$ is a multivariate polynomial that is identically equal to zero,
\small
\begin{align*}
\Omega(x,y) :=
-c_0d_0\lambda &\\
+ (a_1c_1d_0+b_1c_2d_0-a_1c_0d_1-b_1c_0d_2-c_1d_0 \lambda-c_0d_1 \lambda) & x \\
+ (a_2c_1d_0+b_2c_2d_0-a_2c_0d_1-b_2c_0d_2-c_2d_0 \lambda-c_0d_2 \lambda) & y\\
 + (a_3c_1d_0+b_3c_2d_0-a_3c_0d_1+b_1c_2d_1-b_3c_0d_2-b_1c_1d_2-c_1d_1 \lambda) & x^2 \\
  + (a_4c_1d_0+b_4c_2d_0-a_4c_0d_1-a_1c_2d_1+b_2c_2d_1-b_4c_0d_2+a_1c_1d_2-b_2c_1d_2-c_2d_1 \lambda -c_1d_2 \lambda) & x y \\
+ (a_5c_1d_0+b_5c_2d_0-a_5c_0d_1-a_2c_2d_1-b_5c_0d_2+a_2c_1d_2-c_2d_2\lambda) & y^2 \\
 + (b_3c_2d_1-b_3c_1d_2) & x^3 \\
  + (-a_3c_2d_1+b_4c_2d_1+a_3c_1d_2-b_4c_1d_2) & x^2 y \\
 + (-a_4c_2d_1+b_5c_2d_1+a_4c_1d_2-b_5c_1d_2) & x y^2 \\
 + (-a_5c_2d_1+a_5c_1d_2) & y^3 \\
  =& 0.
\end{align*}
\normalsize

\paragraph{Extract algebraic system defining relationship between parameters.}
$\Omega(x,y)$ is identically equal to zero for eigenfunction/eigenvalue pairs of Eq.~\ref{eq:ODE_gen}. 
A polynomial is identically equal to zero if and only if all of its coefficients are zero.  Setting the coefficients of each term $x^n y^m$ of $\Omega(x,y)$ equal to zero produces the system of equations,
\small
\begin{align}\label{eq:system for rational eigenfunctions}
\begin{split} 
    -c_0d_0\lambda & = 0\\
 (a_2c_1d_0+b_2c_2d_0-a_2c_0d_1-b_2c_0d_2-c_2d_0 \lambda-c_0d_2 \lambda) & =0\\
(a_5c_1d_0+b_5c_2d_0-a_5c_0d_1-a_2c_2d_1-b_5c_0d_2+a_2c_1d_2-c_2d_2\lambda) & =0 \\
  (-a_5c_2d_1+a_5c_1d_2) & =0\\
 (a_1c_1d_0+b_1c_2d_0-a_1c_0d_1-b_1c_0d_2-c_1d_0 \lambda-c_0d_1 \lambda) & =0 \\
  (a_4c_1d_0+b_4c_2d_0-a_4c_0d_1-a_1c_2d_1+b_2c_2d_1-b_4c_0d_2+a_1c_1d_2-b_2c_1d_2-c_2d_1 \lambda -c_1d_2 \lambda) & =0 \\
 (-a_4c_2d_1+b_5c_2d_1+a_4c_1d_2-b_5c_1d_2) & =0 \\
  (a_3c_1d_0+b_3c_2d_0-a_3c_0d_1+b_1c_2d_1-b_3c_0d_2-b_1c_1d_2-c_1d_1 \lambda) & =0 \\
  (-a_3c_2d_1+b_4c_2d_1+a_3c_1d_2-b_4c_1d_2) & =0 \\
  (b_3c_2d_1-b_3c_1d_2) & =0.
  \end{split}
\end{align}
\normalsize

Let us call this system, Eqs.~\ref{eq:system for rational eigenfunctions}, $\bm{H}(\bm{a}, \bm{b}, \bm{c}, \bm{d}, \lambda) = \bm{0}$.
This system must be satisfied for $(\lambda, \, \varphi(x,y))$ to be an eigenvalue/eigenfunction pair of Eq.~\ref{eq:ODE_gen}.

We will use tools from nonlinear algebra to study the polynomial system, Eqs.~\ref{eq:system for rational eigenfunctions}, to (1) identify the necessary requirements for a linear rational eigenfunction solution to exist and (2) derive formulas for eigenfunction parameters in settings where a solution does exist.

Our objectives are the following:
\begin{enumerate}
    \item Find the set of ODE parameters $\bm{a}$ and $\bm{b}$
    for which there is at least one nontrivial solution to the system of equations \ref{eq:system for rational eigenfunctions} for some eigenfunction parameters $\bm{c}$ and $\bm{d}$ with some eigenvalue $\lambda$. 
    \item Find the subset of ODE parameters $\bm{a}$ and $\bm{b}$ for which the system of equations \ref{eq:system for rational eigenfunctions} is satisfied by at least two different sets of eigenfunction parameters and eigenvalues that correspond to at least two independent eigenfunctions.  
    \item Find formulas for the eigenfunction parameters in terms of the ODE parameters, given that the ODE parameters belong to the subset where two independent linear rational eigenfunction solutions exist. 
\end{enumerate}

The system of equations $\bm{H}(\bm{a}, \bm{b}, \bm{c}, \bm{d}, \lambda) = \bm{0}$, Eq.~\ref{eq:system for rational eigenfunctions}, is a nonlinear polynomial system in the unknown ODE parameters. Given the equation structure, it likely has a high-dimensional complex solution set.  To solve the system, we take a methodical approach that is well-defined in computational algebraic geometry. We will see that the set of nontrivial solutions can be decomposed as a union of  linear space and a  high-dimensional set, both infinite, whose explicit descriptions we derive. The precise approach for constructing such analytical solutions is described in Section~\ref{sec:alg_geo} and in more detail in Appendix~\ref{sec:alg_geo_appx}. 

Once we have found the subset of quadratic ODEs that possess two independent linear rational eigenfunctions and have computed the eigenfunction parameter values, we can use the obtained eigenfunctions and eigenvalues to construct a solution to the original system, as is presented in Section~\ref{sec:solve_ode}.



\section{Solving quadratic ODEs using independent linear rational eigenfunctions}\label{sec:solve_ode}

Once $n$ independent eigenfunctions are determined for an $n-$dimensional ODE, we can construct analytical solutions for the ODE with the following steps, which we initially formulated in \cite{morrisonSolvingNonlinearOrdinary2024}: 

\begin{enumerate}[noitemsep]
    \item Construct the eigenfunction solutions: $\varphi_j(\bm{x}(t,\bm{x}_0)) = \varphi_j(\bm{x}_0) e^{\lambda_j t}$.
    \item Solve for the original variables as functions of the eigenfunctions: $\bm{x} = \bm{G}(\bm{\varphi})$ where $\bm{\varphi} = (\varphi_1,...,\varphi_n)$.
    \item Substitute the eigenfunction solutions into $\bm{G}$ to construct the analytical solutions:\\ $\bm{x}(t,\bm{x}_0) = \bm{G}(\bm{\varphi}(t, \bm{x}_0))$ where 
    $\bm{\varphi}(t, \bm{x}_0) = (\varphi_1(\bm{x}_0) e^{\lambda_1 t},...,\varphi_n(\bm{x}_0) e^{\lambda_n t})$.
\end{enumerate}

\smallskip

Suppose that we have obtained two independent, nontrivial eigenfunctions for the ODE Eq.~\ref{eq:ODE_gen}, 
\begin{align}
\begin{split}\label{eq:phi1_phi2_gen}
    \varphi_1(x,y) &= \frac{c_0 + c_1 x + c_2 y}{d_0 + d_1 x + d_2 y} \\
    \varphi_2(x,y) &= \frac{k_0 + k_1 x + k_2 y}{m_0 + m_1 x + m_2 y}
\end{split}
\end{align}
where $c_i, d_i, k_i, m_i \in \mathds{C}$ and $\lambda_1, \lambda_2 \in \mathds{C}$ are the corresponding eigenvalues.

Solving the system Eqs.~\ref{eq:phi1_phi2_gen} for $x$ and $y$ as functions of $\varphi_1$ and $\varphi_2$ produces
\begin{align}
\begin{split}\label{eq:x_y_gen}
    x(\varphi_1, \varphi_2) &= -\frac{(c_2 - d_2 \varphi_1)(k_0 - m_0 \varphi_2) - (c_0 - d_0 \varphi_1)(k_2 - m_2 \varphi_2)}{(c_2 - d_2\varphi_1)(k_1 - m_1 \varphi_2) - (c_1 - d_1 \varphi_1)(k_2 - m_2 \varphi_2)} \\
    y(\varphi_1, \varphi_2) &= \frac{(c_1 - d_1 \varphi_1) ( k_0 -  m_0 \varphi_2) - (c_0    - d_0  \varphi_1 )(k_1  - m_1 \varphi_2) }{(c_2   - d_2  \varphi_1 )(k_1  - m_1 \varphi_2)  -(c_1  - d_1 \varphi_1 )( k_2  - m_2 \varphi_2 )}.
\end{split}
\end{align}
Eqs.~\ref{eq:x_y_gen} give the solution to the nonlinear ordinary differential equation in terms of the linear rational eigenfunctions, whose dynamics are easily solved.
Note that if $\varphi_2(x,y) = \alpha \varphi_1(x,y)$ or $\varphi_2(x,y) = \alpha/\varphi_1(x,y)$ for some $\alpha \in \mathds{C}$, the eigenfunctions are not independent \cite{morrisonSolvingNonlinearOrdinary2024, bolltGeometricConsiderationsGood2021}.

The dynamics of the eigenfunctions are linear by construction (Eq.~\ref{eq:dyn_eig}), and therefore have the following solutions as a function of time:
\begin{align}
\begin{split}\label{eq:phi1t_phi2t}
    \varphi_1(t) &= \varphi_1(x_0, y_0) e^{\lambda_1 t} \\
    \varphi_2(t) &= \varphi_2(x_0, y_0) e^{\lambda_2 t}
\end{split}
\end{align}
where initial conditions $\varphi_1(0) = \varphi_1(x_0, y_0)$ and $\varphi_2(0)=\varphi_2(x_0, y_0)$ are the eigenfunctions, Eq.~\ref{eq:phi1_phi2_gen} evaluated at the initial condition in the original state space $(x_0, y_0) = (x(0), y(0))$. Substituting the eigenfunction solutions \ref{eq:phi1t_phi2t} into the equations for $x$ and $y$, Eq.~\ref{eq:x_y_gen}, produces the analytical solution to Eq.~\ref{eq:ODE_gen} for the initial condition $(x_0, y_0)$,
\begin{align}
\begin{split}\label{eq:xt_yt}
    x(t) &= -\frac{(c_2 - d_2 \varphi_1(0) e^{\lambda_1 t} )(k_0 - m_0 \varphi_2(0) e^{\lambda_2 t}) - (c_0 - d_0 \varphi_1(0) e^{\lambda_1 t} )(k_2 - m_2 \varphi_2(0) e^{\lambda_2 t})}{(c_2 - d_2 \varphi_1(0) e^{\lambda_1 t} )(k_1 - m_1 \varphi_2(0) e^{\lambda_2 t}) - (c_1 - d_1 \varphi_1(0) e^{\lambda_1 t} )(k_2 - m_2 \varphi_2(0) e^{\lambda_2 t})}, \\
    y(t) &= \frac{(c_1 - d_1 \varphi_1(0) e^{\lambda_1 t} ) ( k_0 -  m_0 \varphi_2(0) e^{\lambda_2 t}) - (c_0    - d_0  \varphi_1(0) e^{\lambda_1 t}  )(k_1  - m_1 \varphi_2(0) e^{\lambda_2 t}) }{(c_2   - d_2  \varphi_1(0) e^{\lambda_1 t}  )(k_1  - m_1 \varphi_2(0) e^{\lambda_2 t})  -(c_1  - d_1 \varphi_1(0) e^{\lambda_1 t}  )( k_2  - m_2 \varphi_2(0) e^{\lambda_2 t} )}.
\end{split}
\end{align}

%
The most difficult step in this process is finding suitable eigenfunctions.  Once independent eigenfunctions are found, the computations for the remainder of the steps are straightforward.

\section{Identifying solvable quadratic ODEs using nonlinear algebra}
\label{sec:alg_geo}

Now that we have established how eigenfunctions are used to construct solutions to ODEs, we turn our attention back to Objectives (1-3): finding nontrivial linear rational eigenfunctions for quadratic ODEs, or equivalently, finding nontrivial solutions to the system of equations $\bm{H}(\bm{a}, \bm{b}, \bm{c}, \bm{d}, \lambda) = \bm{0}$, Eqs.~\ref{eq:system for rational eigenfunctions}.
These equations mix the ODE’s coefficients with the parameters of the eigenfunctions we hope to find. Because this system is nonlinear and highly coupled, solving it directly by hand is not practical. Instead, we use tools from nonlinear algebra, which provide systematic, exact ways to detect when such systems have meaningful (nontrivial) solutions.

The equations derived from \eqref{eq:dyn_eig} always include \emph{trivial solutions} such as
zero eigenfunctions or zero eigenvalues. These do not yield useful Koopman eigenfunctions and
are therefore removed automatically in the symbolic filtering process. Retaining only the
nontrivial cases is essential because many ODE coefficient choices satisfy the trivial
conditions but produce no usable eigenfunctions. 

To organize this process, we consider the set of all ODE coefficient vectors, $\bm{a}$ and $\bm{b}$, for which the
polynomial system $\bm{H}(\bm{a}, \bm{b}, \bm{c}, \bm{d}, \lambda) = \bm{0}$ admits a \emph{nontrivial} linear rational eigenfunction. We refer to this
set as the \emph{rational eigenfunction variety}; see Definition~\ref{defn:the rational eigenfunction variety}.  
Conceptually, this variety is the region in coefficient space where at least one nontrivial eigenfunction exists. Appendix~\ref{sec:alg_geo_appx} explains how  this set is computed symbolically and how trivial solutions are removed. 

The symbolic computation for the defining equations of the rational eigenfunction variety yields a remarkable classification. Every quadratic ODE in
the normal form~\eqref{eq:ODE_gen} that admits at least one nontrivial linear rational eigenfunction belongs to one of two families:
\begin{itemize}
    \item {Linear family $\mathcal L$}:
    a six-dimensional linear space of coefficients satisfying four linear relations.
    \item {Nonlinear family $\mathcal Q$}:
    a six-dimensional nonlinear space described by nine quadratic relations.
\end{itemize}

\subsection{Singularities in the eigenfunction variety}
    For ODEs in~$\mathcal L$, Section~\ref{subsec:space_L} provides closed-form formulas for     two independent eigenfunctions. 
However, ODEs corresponding to $\mathcal Q$ do not always admit two independent linear rational eigenfunctions. For this \emph{nonlinear} $8$-dimensional subset of $\mathbb C^{10}$, that is, the set of parameters $(\bm{a}, \bm{b})$ that satisfies $\bm{R}(\bm{a}, \bm{b})= \bm{0}$, we consider what it means to have eigenfunction parameter solutions $(\bm{c}, \bm{d}, \lambda)$ of $\bm{H}(\bm{a}, \bm{b},\bm{c}, \bm{d}, \lambda)= \bm{0}$ that correspond to two independent eigenfunction solutions. Consider Fig.~\ref{fig:sings} which shows, on the left, a schematic of the nonlinear variety $\mathcal Q$, points $(\bm{a}, \bm{b}) \in \mathds{R}^{10}$ defined by $\bm{R}(\bm{a}, \bm{b})= \bm{0}$, and on the right, the corresponding points in the space of all parameters $(\bm{a}, \bm{b}, \bm{c}, \bm{d}, \lambda) \in \mathds{C}^{17}$ defined by $\bm{H}(\bm{a}, \bm{b},\bm{c}, \bm{d}, \lambda)= \bm{0}$. 

By definition, for each point $(\bm{a}, \bm{b}) \in \mathds{R}^{10}$ such that $\bm{R}(\bm{a}, \bm{b})= \bm{0}$, there exists at least one corresponding point $(\bm{c}, \bm{d}, \lambda) \in \mathds{C}^7$ such that $\bm{H}(\bm{a}, \bm{b},\bm{c}, \bm{d}, \lambda)= \bm{0}$. 
For some points $(\bm{a}, \bm{b})$ in $\mathcal{Q}$ there exist at least two sets of solutions $(\bm{c}, \bm{d}, \lambda_1)$ and $(\bm{k}, \bm{m}, \lambda_2)$ where $\bm{H}(\bm{a}, \bm{b},\bm{c}, \bm{d}, \lambda_1)= \bm{0}$ and $\bm{H}(\bm{a}, \bm{b},\bm{k}, \bm{m}, \lambda_2)= \bm{0}$ and the points $(\bm{c}, \bm{d}, \lambda_1)$ and $(\bm{k}, \bm{m}, \lambda_2)$ are the parameters for two independent eigenfunctions.

The points on the variety $\mathcal Q$ correspond to ODEs that have at least one linear rational eigenfunction but \textit{two} independent eigenfunctions of this form is not guaranteed and two independent eigenfunctions are necessary to perform the remaining steps and solve the ODE. The points on the variety that lend themselves to 
two independent linear rational eigenfunctions correspond to singularities of the variety, which are described by $\bm{X}(\bm{a}, \bm{b}) = \bm{0}$.  The reasoning is as follows: $\bm{R}(\bm{a}, \bm{b})=\bm{0}$ forms a smooth continuous manifold 
except at points where the Jacobian of $\bm{R}$ drops rank. Such points are, by definition, singularities of the variety defined by $\bm{X}(\bm{a}, \bm{b}) = \bm{0}$.
In Fig.~\ref{fig:sings}(a) we select an arbitrary singular point $(\bm{a}^*, \bm{b}^*)$  on $\bm{X}(\bm{a}, \bm{b})=\bm{0}$ along which the tangent space of $\bm{R}(\bm{a}, \bm{b}) = \bm{0}$ has a higher than expected dimension.
Multiple branches of the variety $\mathcal Q$ intersect at this point. We pick two of these branches and select a point near the singular point on each branch: let $(\bm{a}^{(1)}, \bm{b}^{(1)})$ be on the first branch and let $(\bm{a}^{(2)}, \bm{b}^{(2)})$ be on the second branch. Let the variety have a well defined tangent space at these points.  Now we define a simple smooth continuous path from $(\bm{a}^{(1)}, \bm{b}^{(1)})$ to $(\bm{a}^{*}, \bm{b}^{*})$ along $\bm{R}(\bm{a}, \bm{b})= \bm{0}$ which we call curve 1 and another simple smooth continuous path from $(\bm{a}^{(2)}, \bm{b}^{(2)})$ to $(\bm{a}^{*}, \bm{b}^{*})$ along $\bm{R}(\bm{a}, \bm{b})= \bm{0}$ which we call curve 2 (Fig.~\ref{fig:sings}(a)).
By definition, there exist at least two points $(\bm{c}^{(1)}, \bm{d}^{(1)}, \lambda^{(1)})$ and $(\bm{c}^{(2)}, \bm{d}^{(2)}, \lambda^{(2)})$ such that $\bm{H}(\bm{a}^{(1)}, \bm{b}^{(1)},\bm{c}^{(1)}, \bm{d}^{(1)}, \lambda^{(1)})= \bm{0}$ and $\bm{H}(\bm{a}^{(2)}, \bm{b}^{(2)},\bm{c}^{(2)}, \bm{d}^{(2)}, \lambda^{(2)})= \bm{0}$. These two points most likely correspond to two independent eigenfunctions. Curves 1 and 2 map to corresponding curves in the ODE and eigenfunction parameter space (Fig.~\ref{fig:sings}(b)).
As we move along the curves corresponding to curve 1 and curve 2 in the ODE and eigenfunction space $\bm{H}(\bm{a}, \bm{b},\bm{c}, \bm{d}, \lambda)= \bm{0}$, we find that the ODE parameters approach $(\bm{a}^*, \bm{b}^*)$ along both curves, but there is no guarantee that the eigenfunction parameters will approach the same values.  If
$$\lim_{\underset{\text{along curve 1}}{(\mathbf{a}, \mathbf{b}) \longrightarrow (\mathbf{a}^*, \mathbf{b}^*)}} (\mathbf{a}, \mathbf{b}, \mathbf{c}, \mathbf{d}, \lambda) \neq \lim_{\underset{\text{along curve 2}}{(\mathbf{a}, \mathbf{b}) \longrightarrow (\mathbf{a}^*, \mathbf{b}^*)}} (\mathbf{a}, \mathbf{b}, \mathbf{c}, \mathbf{d}, \lambda)$$

then we will achieve two sets of eigenfunction parameter solutions for a single ODE parameter point $(\bm{a}^*, \bm{b}^*)$ and can use these to construct two independent eigenfunctions.
Fig.~\ref{fig:sings}(b) shows with a cartoon how the two curves mapped to the ODE and eigenfunction space are not guaranteed to approach the same point and therefore should yield two independent eigenfunctions.
For ODEs corresponding to the points in $\mathcal X$, Section~\ref{subsec:space_X} provides explicitly two independent linear rational eigenfunctions.

\begin{figure}[th!]
    \centering
    \includegraphics[width=1.0\textwidth]{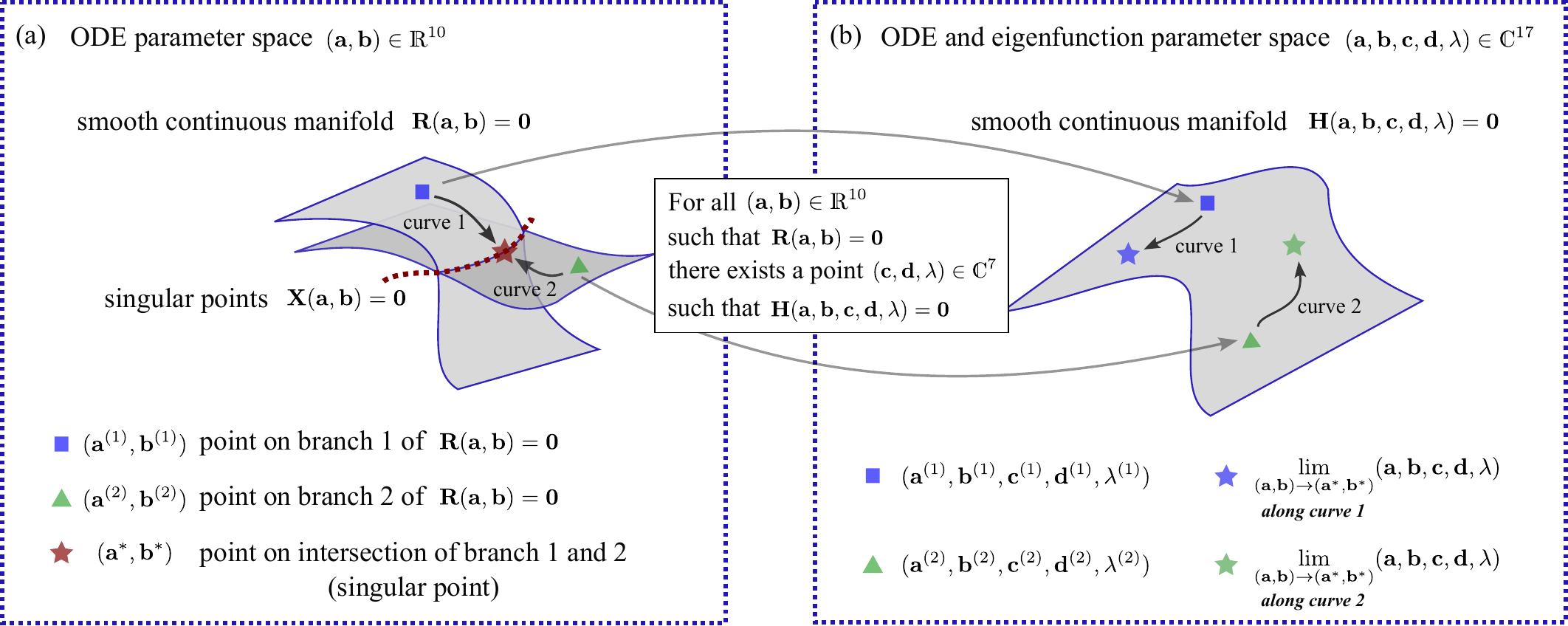}
    \caption{(a). Singularities of $\bm{R}(\bm{a}, \bm{b})=\bm{0}$ are described by 
    $\bm{X}(\bm{a}, \bm{b}) = \bm{0}$. (b) Points on $\bm{R}(\bm{a}, \bm{b})= \bm{0}$ map to points on $\bm{H}(\bm{a}, \bm{b}, \bm{c}, \bm{d}, \lambda) = \bm{0}$. Points along $\bm{X}(\bm{a}, \bm{b}) = \bm{0}$ can map to two sets of eigenfunction parameter values.}
    \label{fig:sings}
\end{figure}

The result therefore is a subspace $\mathcal{V} = \mathcal{L} \cup \mathcal{X}$ that contains all sets of Eq.~\ref{eq:ODE_gen} parameters 
$$\{a_1, a_2, a_3, a_4, a_5, b_1, b_2, b_3, b_4, b_5\}$$ for which we are guaranteed two independent linear rational eigenfunctions.
Thus, the question ``Which quadratic ODEs are solvable by the eigenfunction method using linear rational eigenfunctions?'' has a clear answer: \emph{exactly those whose coefficients lie in $\mathcal L$ or $\mathcal X$}.   These families are the two components extracted from the rational eigenfunction variety. 


\section{Solvable quadratic ODEs and their linear rational eigenfunctions}
\label{sec:solvable_ODEs_eig_params}

There are two spaces, $\mathcal L$ and $\mathcal X$, where Eq.~\ref{eq:ODE_gen} has independent linear rational eigenfunctions. 

$\mathcal L$ is a $6$-dimensional \emph{linear} space defined by the $4$ linear polynomials\footnote{Meaning, the constraints are obtained by setting these linear polynomials to $=0$.} 
\[ b_3,\quad a_5,\quad a_4-b_5,\quad a_3-b_4.
\]

$\mathcal X$ is a nonlinear space defined by the quadratic polynomials\footnote{Meaning, the constraints are obtained by setting these polynomials to $=0$.} 
\begin{align*}
     4a_5b_3-a_4b_4,            \\
 2a_4b_3-2a_3b_4+b_4^2-4b_3b_5,  \\
      2a_2b_3-a_1b_4+b_2b_4-2b_1b_5,  \\
      2a_5b_1-a_2b_4,                \\
       a_4b_1-a_1b_4+b_2b_4-2b_1b_5,   \\
      2a_3b_1-2a_1b_3+2b_2b_3-b_1b_4, \\
      a_4^2-4a_3a_5+2a_5b_4-2a_4b_5,  \\
      a_2a_4-2a_1a_5+2a_5b_2-2a_2b_5, \\
      2a_2a_3-a_1a_4+a_4b_2-a_2b_4 .  
\end{align*}

\subsection{Linear space \texorpdfstring{$\mathcal{L}$}{L} ODEs and eigenfunctions}\label{subsec:space_L}

The linear space $\mathcal L$ produces quadratic ODEs of the form  
\begin{align}
\begin{split}\label{eq:ODE_lin}
    \frac{dx}{dt} &= a_1 x + a_2 y + x(b_4 x + a_4 y) \\
    \frac{dy}{dt} &= b_1 x + b_2 y + y( b_4 x + a_4 y) 
    \end{split}
\end{align}
with real parameters, $a_i, b_i \in \mathds{R}$.

Two independent eigenvalue / eigenfunction pairs for this family are
\begin{align*}
    \Big( \lambda_1 &= \frac{a_1 + b_2 - \Delta}{2}, \quad \varphi_1(x,y) = \frac{x + c_2 y}{d_0 + x + d_2 y} \Big) \quad \text{and} \\
    \Big( \lambda_2 &= \Delta, \quad \varphi_2(x,y) = \frac{x + k_2 y}{x + m_2 y} \Big)
\end{align*}
where 
$\Delta = \sqrt{a_1^2+4a_2 b_1 - 2a_1 b_2 +b_2^2}$ 
and the eigenfunction parameter values are

\begin{align*}
&c_2= \frac{-a_1+b_2 - \Delta}{2b_1}, \;
d_0= \frac{a_2 b_1 - a_1 b_2}{a_4 b_1 - b_2 b_4}, \;
d_2= \frac{-a_1 a_4 + a_2 b_4}{a_4 b_1 - b_2 b_4}, \\
&k_2 = \frac{-a_1 + b_2 + \Delta}{2 b_1}, \;
    m_2 = \frac{-a_1 + b_2 - \Delta}{2 b_1}.
\end{align*}
\normalsize


Note that any multiple of an eigenfunction is also an eigenfunction in the same equivalence class \cite{morrisonSolvingNonlinearOrdinary2024},
\begin{align*}
    \varphi_1^{(\alpha)}(x,y) = \alpha \varphi_1(x,y).
\end{align*}
Given that any multiple of an eigenfunction is an eigenfunction in the same equivalent class, we have normalized the two independent eigenfunctions by setting $c_1 = d_1 = 1$ and $k_1 = m_1 = 1$.

Observe that the eigenvalues and eigenfunction parameters are functions of the ODE parameters (the ODE parameters are $\{ a_1, a_2, a_4, b_1, b_2, b_4 \}$); the eigenfunctions are determined directly from the ODE parameters. 
We can use these eigenfunctions to construct solutions to the original system as outlined in Section~\ref{sec:solve_ode}.

\subsection{Nonlinear space \texorpdfstring{$\mathcal{X}$}{X} ODEs and eigenfunctions}\label{subsec:space_X}

The nonlinear space $\mathcal X$ produces quadratic ODEs of the form 
\begin{align}
\begin{split}\label{eq:ODE_nonlin}
    \frac{dx}{dt} &= a_1 x + a_2 y + a_4 xy + \frac{a_1 a_4 -b_2 a_4 +a_2 b_4}{2a_2} x^2 + \frac{a_2 b_4}{2 b_1} y^2\\
    \frac{dy}{dt} &= b_1 x + b_2 y + b_4 xy + \frac{b_1 a_4}{2a_2} x^2 + \frac{b_1 a_4 -a_1 b_4+ b_2 b_4}{2b_1} y^2
\end{split}
\end{align}
with real parameters, $a_i, b_i \in \mathds{R}$.

Two independent eigenvalue / eigenfunction pairs for this family are
\begin{align*}
    \Big( \lambda_1 &= \frac{-a_1-b_2-\Delta}{2}, \quad \varphi_1(x,y) = \frac{c_0 +  x + c_2 y}{ x + d_2 y} \Big) \quad \text{and} \\
    \Big( \lambda_2 &= \frac{-a_1-b_2+\Delta}{2}, \quad \varphi_2(x,y) = \frac{k_0 +  x + k_2 y}{ x + m_2 y} \Big)
\end{align*}
where
$\Delta = \sqrt{a_1^2+4a_2 b_1 - 2a_1 b_2 +b_2^2}$ and the eigenfunction parameter values are
\begin{align*}
c_0&= \frac{ -2 a_2 b_1 a_4 + b_2 a_4 \left( a_1 - b_2 - \Delta \right) + a_2 b_4 \left( a_1 + b_2 + \Delta \right)}
{-b_1 a_4^2 + b_4 \left( a_1 a_4 - b_2 a_4 + a_2 b_4 \right)},\\
c_2  &= d_2 = \frac{ -a_1 + b_2 + \Delta }{2b_1} 
\end{align*}
for the first eigenfunction and
\begin{align*}
    k_0 &= \frac{ -2 a_2 b_1 a_4 + b_2 a_4 \left( a_1 - b_2 + \Delta \right) + a_2 b_4 \left( a_1 + b_2 - \Delta \right) }
{-b_1 a_4^2 + b_4 \left( a_1 a_4 - b_2 a_4 + a_2 b_4 \right)},\\
k_2 &= m_2 = \frac{ -a_1 + b_2 - \Delta  }{2b_1} 
\end{align*}
for the second eigenfunction. Given that any multiple of an eigenfunction is an eigenfunction in the same equivalent class, we normalize the two independent eigenfunctions by setting $c_1 = d_1 = 1$ and $k_1 = m_1 = 1$.

Observe that the eigenvalues and eigenfunction parameters are functions of the ODE parameters (the ODE parameters are $\{ a_1, a_2, a_4, b_1, b_2, b_4 \}$); the eigenfunctions can be determined directly from the ODE parameters. Section~\ref{sec:solve_ode} outlines the method to construct solutions using the eigenfunctions.

\section{Examples}\label{sec:examples}

We now highlight a few ODEs that have parameter values in $\mathcal L \cup X$  and solve the system of nonlinear ODEs using their linear rational eigenfunctions.

\subsection{ODE with parameters in \texorpdfstring{$\mathcal{L}$}{L}}

Consider the ODE,
\begin{align}
\begin{split}\label{eq:L_ex1_ode}
    \frac{dx}{dt} &= x -2y + 2x^2 +xy\\
    \frac{dy}{dt} &= 3x +y +2xy + y^2
    \end{split}
\end{align}
shown in Figure~\ref{fig:L_example1}(a). The parameter values for this ODE, 
$$(a_1, a_2, a_3, a_4, a_5, b_1, b_2, b_3, b_4, b_5) = (1, -2, 2, 1, 0, 3, 1, 0, 2, 1),$$
are in space $\mathcal L$.
Two independent eigenvalues/eigenfunction pairs for Eqs.~\ref{eq:L_ex1_ode} using the formulas from Section~\ref{subsec:space_L}
are 
\begin{align*}
    \lambda_1 = 1 - i\sqrt{6}, \quad \varphi_1(x,y) &= \frac{x-i\sqrt{\frac{2}{3}}y}{-7+x-5y},\\
    \lambda_2 = 2i \sqrt{6}, \quad \varphi_2(x,y) &= \frac{3x + i \sqrt{6}y}{3x - i \sqrt{6}y}.
\end{align*}
Consider the initial condition $(x_0, y_0) = (0, 1)$. The eigenfunction initial conditions are $\varphi_1(x_0, y_0) = \varphi_1(0, 1) = \frac{i}{6 \sqrt{6}}$ and $\varphi_2(x_0, y_0) = \varphi_2(0, 1) = -1$.
The solutions in the eigenfunction space are
$\varphi_1(t) = \frac{i}{6 \sqrt{6}} e^{(1-i\sqrt{6})t}$
and
$\varphi_2(t) = -e^{2i \sqrt{6}t}$ (Fig.~\ref{fig:L_example1}(b)).
Finally, using Eqs.~\ref{eq:x_y_gen} and \ref{eq:xt_yt}, we determine the solution to the initial value problem in the original state space to be
\begin{align}
\begin{split}\label{eq:L_ex1_x_y}
    x(t) &= \frac{14 e^t \sin{
  \sqrt{6} t} (-12 \sqrt{6} + 
   e^t (5 \sqrt{6} \cos{\sqrt{6} t} + 2 \sin{\sqrt{6} t}))}{864 - 48 e^t (15 \cos{\sqrt{6} t} + \sqrt{6} \sin{\sqrt{6} t}) + 
 e^{2 t} (77 + 73 \cos{2 \sqrt{6} t} + 10 \sqrt{6} \sin{2 \sqrt{6} t})},\\
 y(t) &= - \frac{14 e^t \cos{
  \sqrt{6} t} (-36 + 
   e^t (15 \cos{\sqrt{6} t} + \sqrt{6} \sin{\sqrt{6} t}))}{864 - 48 e^t (15 \cos{\sqrt{6} t} + \sqrt{6} \sin{\sqrt{6} t}) + 
 e^{2 t} (77 + 73 \cos{2 \sqrt{6} t} + 10 \sqrt{6} \sin{2 \sqrt{6} t})}.
 \end{split}
\end{align}

Figure~\ref{fig:L_example1}(c) shows the exact solution for initial condition $(0,1)$, Eqs.~\ref{eq:L_ex1_x_y}, compared to the numerical solution computed using NDSolve in \textit{Mathematica}. The analytical and numerical solutions match as expected.  We can use the analytical solution to solve for the finite escape time of this initial value problem by setting the denominators of Eqs.~\ref{eq:L_ex1_x_y} to zero. The singularity occurs approximately at time $t = 2.10895$.

\begin{figure}[th!]
    \centering
    \includegraphics[width=0.8\linewidth]{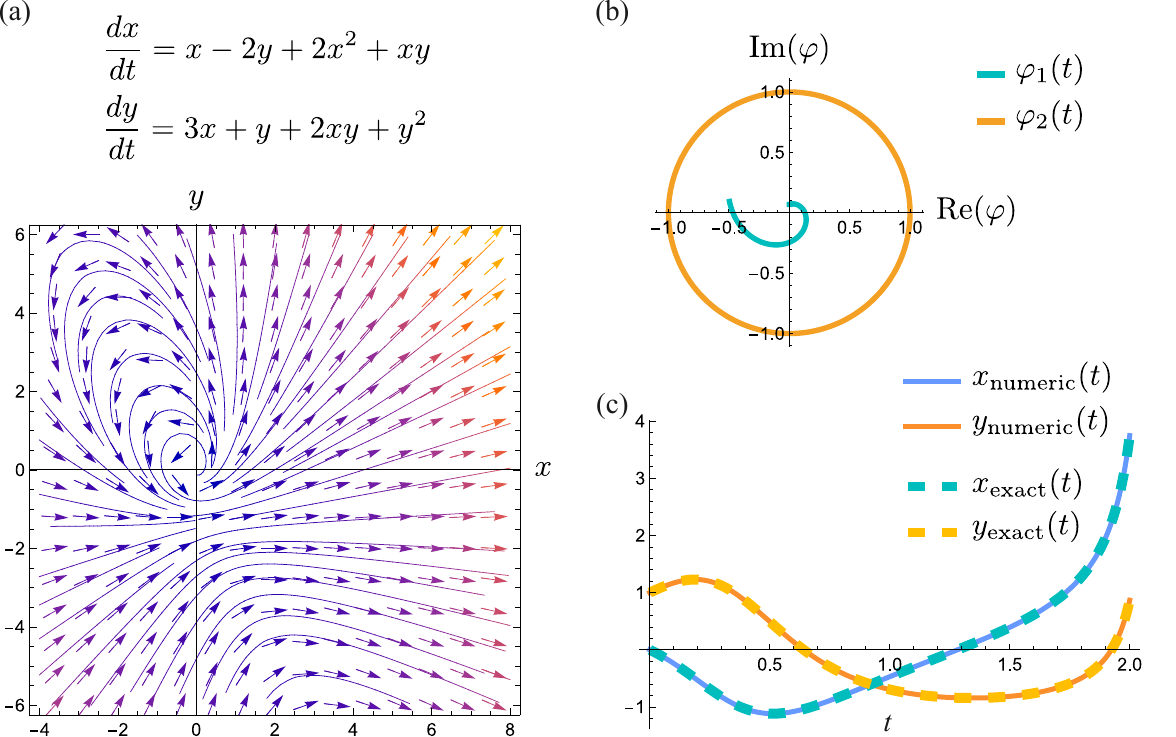}
    \caption{(a) Space $\mathcal L$ dynamical system. (b) Solutions in eigenfunction space for initial condition $(x_0, y_0) = (0,1)$. (c) Analytical versus numerical solutions for the initial condition $(x_0, y_0)=(0,1)$.}
    \label{fig:L_example1}
\end{figure}

\subsection{ODE with parameters in \texorpdfstring{$\mathcal{X}$}{X}: example 1}

Consider the ODE,

\begin{align}
\begin{split}\label{eq:ex1_ode}
    \frac{dx}{dt} &= -4x -2y + x^2 - \frac{2}{3}y^2\\
    \frac{dy}{dt} &= 3x +y +2xy + \frac{5}{3}y^2
    \end{split}
\end{align}
shown in Figure~\ref{fig:example1}(a).
The parameter values for this ODE, 
$$(a_1, a_2, a_3, a_4, a_5, b_1, b_2, b_3, b_4, b_5) = (-4, -2, 1, 0, -\frac{2}{3}, 3, 1, 0, 2, \frac{5}{3}),$$ are in space $\mathcal X$.
Two independent eigenvalue/eigenfunction pairs for Eqs.~\ref{eq:ex1_ode} using the formulas from Section~\ref{subsec:space_X} are
\begin{align*}
    \lambda_1 = 1, \quad \varphi_1(x,y) &= \frac{-1+x+y}{x+y},\\
    \lambda_2 = 2, \quad \varphi_2(x,y) &= \frac{-6+3x+2y}{3x+2y}.
\end{align*}
Consider the initial condition $(x_0, y_0) = (-3, -4)$. This initial conditions in the eigenfunction spaces are $\varphi_1(x_0, y_0) = \varphi_1(-3, -4) = \frac{8}{7}$ and $\varphi_2(x_0, y_0) = \varphi_2(-3, -4) = \frac{23}{17}$. 
From Eqs.~\ref{eq:phi1t_phi2t}, the solutions in the eigenfunction space are $\varphi_1(t) =\frac{8}{7} e^{t} $ and $\varphi_2(t) =\frac{23}{17} e^{2t} $ (Fig.~\ref{fig:example1}(b)).
Using Eqs.~\ref{eq:x_y_gen} and \ref{eq:xt_yt} we determine the solution to the initial value problem to be
\begin{align}
\begin{split}\label{eq:ex1_x_y}
    x(t) &= \frac{476-816e^t + 322 e^{2t}}{119-136e^t - 161 e^{2t} + 184 e^{3t}}, \\
    y(t) &= \frac{-357 + 816e^t - 483 e^{2t}}{119 - 136 e^t - 161 e^{2t} + 184 e^{3t}}.
    \end{split}
\end{align}
Figure~\ref{fig:example1}(c) shows the exact, analytical solution for initial condition $(-3,-4)$, Eqs.~\ref{eq:ex1_x_y}, compared to the numerical solution computed using the numerical differential equation solver NDSolve in \textit{Mathematica}.  The analytical and numerical solutions match as expected.

\begin{figure}[th!]
    \centering
    \includegraphics[width=0.9\linewidth]{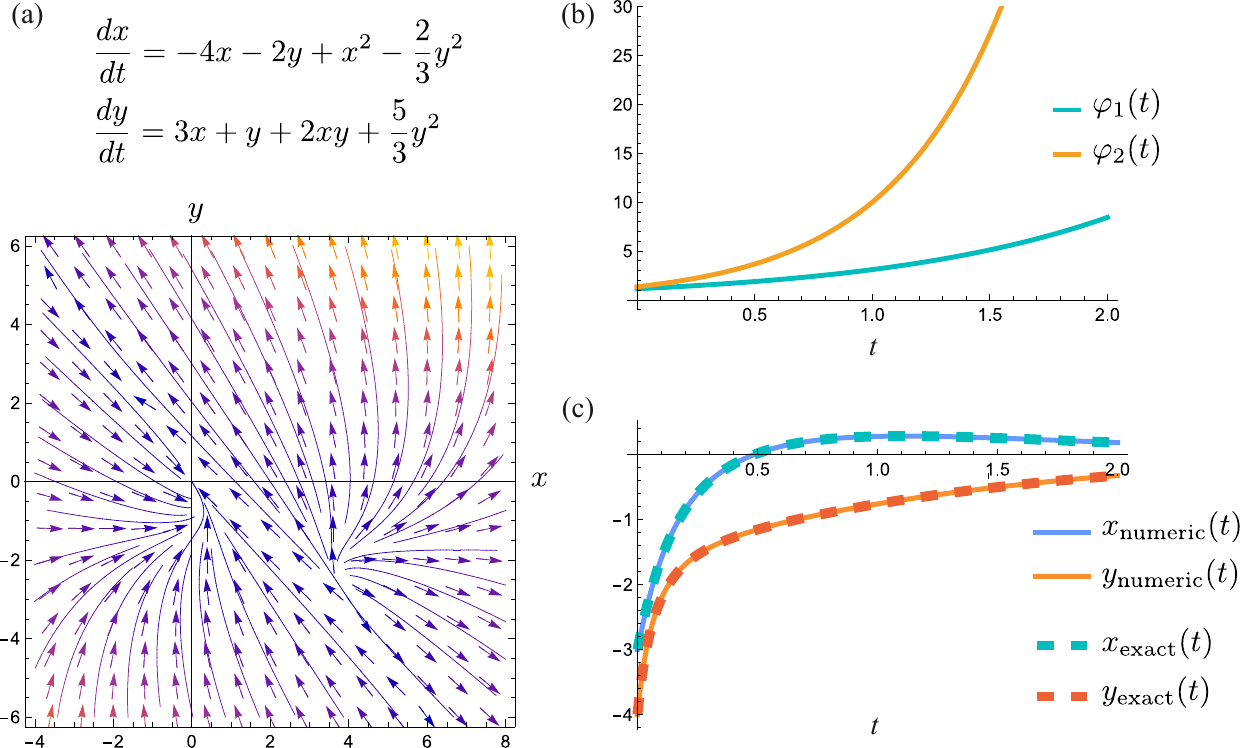}
    \caption{(a) Space $\mathcal X$ example 1 dynamical system. (b) Solutions in eigenfunction space for initial condition $(x_0, y_0) = (-3,-4)$. (c) Analytical versus numerical solutions for the initial condition $(x_0, y_0) = (-3,-4)$.}
    \label{fig:example1}
\end{figure}

\subsection{ODE with parameters in \texorpdfstring{$\mathcal{X}$}{X}: example 2}
Consider the ODE,
\begin{align}
\begin{split}\label{eq:ex2_ode}
    \frac{dx}{dt} &= -3x -2y + x^2 - \frac{2}{3}y^2\\
    \frac{dy}{dt} &= 3x +y +2xy + \frac{4}{3}y^2
    \end{split}
\end{align}
shown in Figure~\ref{fig:example2}(a).  The parameter values for this ODE,
$$(a_1, a_2, a_3, a_4, a_5, b_1, b_2, b_3, b_4, b_5) = (-3, -2, 1, 0, -\frac{2}{3}, 3, 1, 0, 2, \frac{4}{3}),$$ are in space $\mathcal X$.
The eigenvalues and eigenfunctions for Eqs.~\ref{eq:ex2_ode} are 
\begin{align*}
    \lambda_1 = 1-i \sqrt{2}, \quad \varphi_1(x,y) &= \frac{-3+3i\sqrt{2} + 3x+(2+i\sqrt{2})y}{3x+ (2+i\sqrt{2})y},\\
    \lambda_2 = 1+i \sqrt{2}, \quad \varphi_2(x,y) &= \frac{-3-3i\sqrt{2} + 3x+(2-i\sqrt{2})y}{3x+ (2-i\sqrt{2})y}.
\end{align*}
The initial condition $(x_0, y_0)=(2, -1)$ in the eigenfunction space is $\varphi_1(x_0, y_0) = \frac{i-2\sqrt{2}}{4i+\sqrt{2}}$ and $\varphi_2(x_0, y_0) = \frac{i+2\sqrt{2}}{4i-\sqrt{2}}$.
The solutions in the eigenfunction space are
$\varphi_1(t) =\frac{i-2\sqrt{2}}{4i+\sqrt{2}} e^{(1-i\sqrt{2})t} $ 
and $\varphi_2(t) = \frac{i+2\sqrt{2}}{4i-\sqrt{2}} e^{(1+i\sqrt{2})t} $ (Fig.~\ref{fig:example2}(b)).
The solution in the original state space is
\begin{align}
\begin{split}\label{eq:ex2_x_y}
    x(t) &= \frac{6(\sqrt{2} - e^t \sin{\sqrt{2}t})}{\sqrt{2}(2+e^{2t}) - 4 e^t \sin{\sqrt{2}t}},\\
    y(t) &= \frac{-6 \sqrt{2} + 3 e^t (\sqrt{2} \cos{\sqrt{2} t} + 2 \sin{\sqrt{2}t})}{\sqrt{2}(2+e^{2t}) - 4 e^t \sin{\sqrt{2}t}}.
    \end{split}
\end{align}

Figure~\ref{fig:example2}(c) shows the analytical solution for initial condition $(2,-1)$, Eqs.~\ref{eq:ex2_x_y}, compared to the numerical solution computed using NDSolve in \textit{Mathematica}.  The analytical and numerical solutions match as expected.

\begin{figure}[th!]
    \centering
    \includegraphics[width=0.9\linewidth]{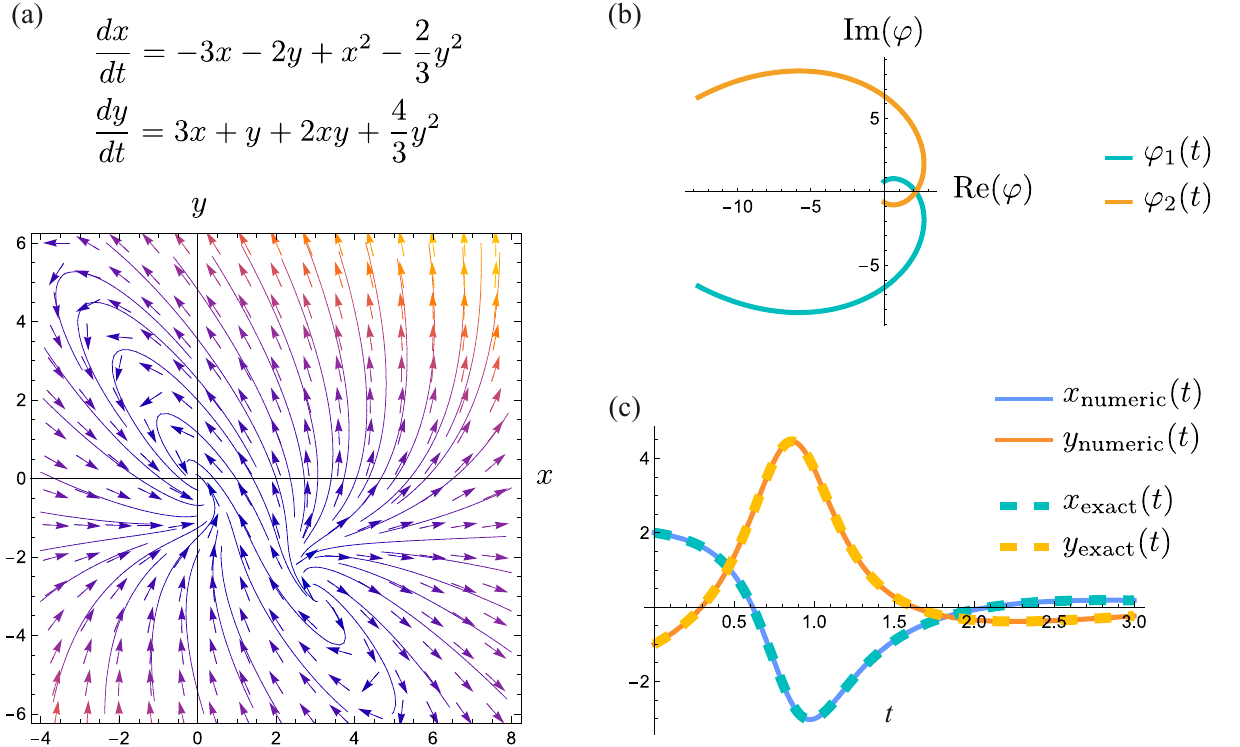}
    \caption{(a) Space $\mathcal X$ example 2 dynamical system. (b) Solutions in eigenfunction space for initial condition $(2,-1)$. (c) Analytical versus numerical solutions for the initial condition $(x_0, y_0) = (2,-1)$.}
    \label{fig:example2}
\end{figure}

\subsection{ODE with parameters in \texorpdfstring{$\mathcal{X}$}{X}: example 3}

Consider the ODE,
\begin{align}
\begin{split}\label{eq:ex3_ode}
    \frac{dx}{dt} &= -x +2y + x^2 +xy - \frac{3}{2}y^2\\
    \frac{dy}{dt} &= -2x +y -\frac{x^2}{2} +3xy -y^2
    \end{split}
\end{align}
shown in Figure~\ref{fig:example3}(a). The parameter values for this ODE, 
$$(a_1, a_2, a_3, a_4, a_5, b_1, b_2, b_3, b_4, b_5) = (-1, 2, 1, 1, -\frac{3}{2}, -2, 1, -\frac{1}{2}, 3, -1),$$ are in space $\mathcal X$.
The eigenvalues and eigenfunctions for Eqs.~\ref{eq:ex3_ode} are 
\begin{align*}
    \lambda_1 = -i \sqrt{3}, \quad \varphi_1(x,y) &= \frac{6+i10\sqrt{3} + 14x -(7 +i7\sqrt{3})y}{14x -(7 +i7\sqrt{3})y},\\
    \lambda_2 = i \sqrt{3}, \quad \varphi_2(x,y) &= \frac{6-i10\sqrt{3} +14x -  (7-i7\sqrt{3})y}{14x - (7-i7\sqrt{3})y}
\end{align*}
Consider the initial condition $(x_0, y_0) = (2, -1)$. This initial conditions in the eigenfunction spaces are $\varphi_1(x_0, y_0) =  \frac{41i-17\sqrt{3}}{35i-7\sqrt{3}}$ and $\varphi_2(x_0, y_0) = \frac{41i+17\sqrt{3}}{35i+7\sqrt{3}}$. 
From Eqs.~\ref{eq:phi1t_phi2t}, the solutions in the eigenfunction space are $\varphi_1(t) =\frac{41i-17\sqrt{3}}{35i-7\sqrt{3}} e^{-i\sqrt{3}t} $ and $\varphi_2(t) =\frac{41i+17\sqrt{3}}{35i+7\sqrt{3}} e^{i\sqrt{3}t} $ (Fig.~\ref{fig:example3}(b)).
The solution in the original state space is
\begin{align}
\begin{split}\label{eq:ex3_x_y}
    x(t) &= \frac{1}{64} \left(-5 + \frac{-798\sqrt{3} + 5733 \sin{\sqrt{3}t}}{-70\sqrt{3} + 64\sqrt{3} \cos{\sqrt{3}t} + 33 \sin{\sqrt{3}t}} \right), \\
    y(t) &= \frac{1}{64} \left( 41 + \frac{21(30 \sqrt{3} + 91 \sin{\sqrt{3}t})}{-70\sqrt{3} + 64\sqrt{3} \cos{\sqrt{3}t} + 33 \sin{\sqrt{3}t}} \right).
    \end{split}
\end{align}
Figure~\ref{fig:example3}(c) shows the analytical solution for initial condition $(2, -1)$, Eqs.~\ref{eq:ex3_x_y}, compared to the numerical solution computed using NDSolve in \textit{Mathematica}. The analytical and numerical solutions match well.

These examples exhibit the variety in dynamics displayed by the quadratic ODE families that are solvable using the rational eigenfunction method with linear rational eigenfunctions outlined above.  Similar to analytical solutions to linear ODEs, the solutions to these nonlinear ODEs are combinations of exponential, sine, and cosine functions.

\begin{figure}[th!]
    \centering
    \includegraphics[width=0.9\linewidth]{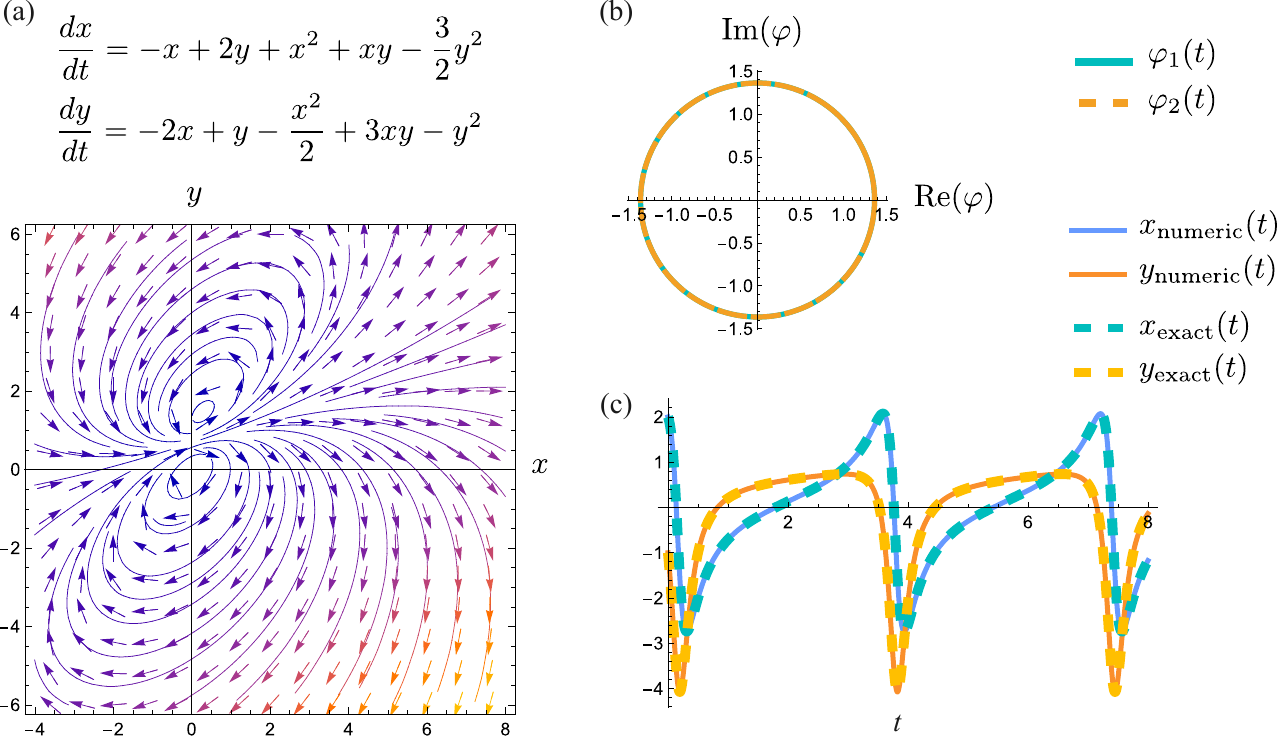}
    \caption{(a) Space $\mathcal X$ example 3 dynamical system. (b) Solutions in eigenfunction space for initial condition $(2,-1)$. (c) Analytical versus numerical solutions for the initial condition $(2,-1)$.}
    \label{fig:example3}
\end{figure}

\section{Discussion}
\label{sec:discussion}

Polynomial ordinary differential equations govern critical systems across biology, engineering, physics, and economics \cite{strogatzNonlinearDynamicsChaos2019, braunDifferentialEquationsTheir1978, porterOverviewPolynomicSystem1976}. Exact solutions are useful for tasks such as global control design and locating limit cycles \cite{isidoriNonlinearControlSystems1990, camachoModelPredictiveControl2007}
Yet there are currently no general methods to derive \emph{closed-form} Koopman eigenfunctions --- an unmet need that we take a step toward addressing by outlining a method to find eigenfunctions with a rational form.
We have established one of the first frameworks for analytically solving classes of quadratic ODEs. Analytical solutions allow for quantitative analysis and a reduced computational burden.
This work synthesizes nonlinear algebra, computational algebraic geometry, and dynamical systems theory into a coherent pipeline, with preliminary results demonstrating feasibility and setting new technical benchmarks.

\smallskip

\subsection{Comparison to other methods}

Data-driven Koopman methods such as EDMD and neural variants excel at computing \emph{approximate} eigenfunctions, which have proven useful for analyzing ODEs and designing control procedures \cite{mauroyGlobalStabilityAnalysis2016,williamsDataDrivenApproximation2015,kutzDynamicModeDecomposition2016,kaiserDatadrivenDiscoveryKoopman2021,liExtendedDynamicMode2017,williamsKernelbasedMethodDatadriven2015,bruntonModernKoopmanTheory2021, kvalheimExistenceUniquenessGlobal2021}.
However, these methods do not deliver \emph{closed-form} eigenfunctions—and thus cannot certify global, exact solutions, especially near singular behavior. Our approach instead produces \emph{exact} eigenfunctions by converting the eigenfunction equation into a structured polynomial system and analyzing its solution set with nonlinear algebra. This yields:
(i) explicit solvability conditions in the ODE’s coefficient space (the newly introduced \emph{rational eigenfunction variety}); 
(ii) principled detection of eigenfunction multiplicity via singular-locus geometry; and 
(iii) closed-form parameter formulas when solvable, precisely the ingredients required for analytic solutions.

While data-driven methods have met with much success, there are other approaches to solve nonlinear ODEs.
Power series methods posit a power series solution to the IVP and coefficients are determined recursively \cite{barlowPowerSeriesSolutions2025}. Unfortunately, power series solutions have some of the same drawbacks as EDMD --- solutions are approximate, highly inaccurate outside of certain regions, and an infinite number of terms is usually required for convergence \cite{barlowPowerSeriesSolutions2025}.

Rational functions are useful basis functions in other contexts, such as representing reduced order models of nonlinear systems and for approximating invariant manifolds
\cite{kleinEntropystableModelReduction2025, kaszasGlobalizingManifoldbasedReduced2025}.  Rational functions are used as a basis for approximation in contexts other than dynamical systems. For example, non-uniform rational B-spline (NURBS) curves and surfaces are used widely for geometric representation and design \cite{pieglNURBSSurvey1991, pieglNURBSBook1997}.  Our results indicate that rational functions may be broadly useful basis functions for finding eigenfunctions of polynomial ODEs.

\smallskip

\subsection{Extensions to broader classes of ODEs and eigenfunctions}

While we applied our method to two-dimensional, quadratic ODE and sought eigenfunctions with a linear rational form, the underlying process is not inherently limited to this narrow case.  A natural extension of this work would be to apply the same method to quadratic ODEs but allow the eigenfunctions to have a more general rational form, such as quadratic polynomials in the numerator and denominator rather than linear polynomials. Ref.~\cite{morrisonSolvingNonlinearOrdinary2024} shows two examples of quadratic ODEs with quadratic rational eigenfunctions; allowing the eigenfunctions to have a more general quadratic rational form should expand the set of quadratic ODEs that can be solved using this method.

Similarly, the polynomial ODEs could also be made more general to include higher order terms such as cubic terms.  The ODEs could also be generalized to higher dimensions, such as three variable systems of ODEs.  More interesting dynamics, such as chaos, appear in three dimensional ODEs and this approach for finding eigenfunctions may provide a useful avenue for analyzing the increasingly complex dynamics that can appear in higher dimensional systems \cite{strogatzNonlinearDynamicsChaos2019}. Methods from computational algebraic geometry are able to solve high-dimensional systems and so the limitations to the generalization of this process is not yet clear.

\section{Conclusions}
\label{sec:conclusions}

We employ methods from computational algebraic geometry to find eigenfunctions that have a closed-form for certain families of two-dimensional quadratic ODEs.
By positing a linear rational form for the eigenfunctions in the eigenfunction equation, we derive a system of polynomial equations that describe the constraints on the ODE parameters and eigenfunction parameters in tandem.  By applying methods from computational algebraic geometry, we solved for the ODE families that admit linear rational eigenfunctions and derived explicit formulas for the eigenfunctions and eigenvalues. We then used the eigenfunctions to construct analytical solutions to the ODEs.
We demonstrated our method on several example ODEs, showing that the analytical solutions match numerical solutions. This work provides a new framework for analytically solving some classes of polynomial ODEs and has the potential to be extended to more general classes of ODEs and eigenfunctions.


\appendix

\section{Normal form for quadratic ODEs}\label{sec:normal_form}
Note that Eqs.~\ref{eq:ODE_gen} do not contain added constants. However, any quadratic ODE containing added constants, 
\begin{align}
\begin{split}\label{eq:ODE_gen_wconsts}
    \frac{du}{dt} &= a_0 +  a_1 u + a_2 v + a_3 u^2 + a_4 uv + a_5 v^2,\\
    \frac{dv}{dt} &= b_0 + b_1 u + b_2 v + b_3 u^2 + b_4 uv + b_5 v^2,
    \end{split}
\end{align}
can be transformed into the normal form of Eqs.~\ref{eq:ODE_gen} using the change of variables $u = \frac{a_1 x - a_0}{a_1}$, and $v = \frac{b_1 y - b_0}{b_1}$.  In this way, quadratic ODEs with constants, Eqs.~\ref{eq:ODE_gen_wconsts}, may be converted to normal form, Eqs.~\ref{eq:ODE_gen}, solved, if possible, and then transformed back to the original variables.

\section{Details on the nonlinear algebra computations}\label{sec:alg_geo_appx}

This section offers a more detailed description of the method used to solve the system of equations~\ref{eq:system for rational eigenfunctions}.  The code and polynomials resulting from it are available on the GitHub repository for this paper. 

\subsection{The rational eigenfunction variety of an ODE system}

Finding restrictions that  coefficients must satisfy in order for  a polynomial system to have a solution can be done using computational algebraic geometry. To explain the computation, we need to introduce some vocabulary.   The power of algebraic geometry is the  connection between ideals generated by polynomial systems and the geometric objects which are their common zeros, the connection we only state at a high level and without the technical details.  The formalism behind algebra-geometry dictionary is summarized in, for example, Chapter 4 of the standard textbook \cite{CLO}. 

Let $\bm{H}(\bm{a}, \bm{b}, \bm{c}, \bm{d}, \lambda) = (f_1,\dots,f_{10})$ be the ten polynomials in the system in Equation~\eqref{eq:system for rational eigenfunctions}; that is, $f_1=-c_0d_0\lambda$, and so on.  Since we are interested  on possible constraints on $a_i$'s and $b_i$'s so that  a nontrivial solution  $(c_0,c_1,c_2,d_0,d_1,d_2)$ exists, we treat $a_i$'s and $b_i$'s also as indeterminates,  start by seeking a joint solution for all $17$ unknowns, and eliminate the unwanted unknowns later. 
To this end, the $10$ defining polynomials are considered in a polynomial ring  with $17$ variables, which in algebraic notation is denoted by \[R:=\mathbb C[\lambda,c_0,c_1,c_2,d_0,d_1,d_2,a_1,\dots,a_5,b_1,\dots,b_5].\] 
Define $I\subset R$ to be the \emph{ideal} generated by the polynomials (cf.\  Definition 2 in \S 1.4 of \cite{CLO}):  
\[ I:=\langle f_1,\dots,f_{10}\rangle :=\{\sum_{i=1}^{10}h_if_i : h_i\in R\} \subset R.\]
Crucially, by definition of an ideal, each point $P\in\mathbb C^{17}$ which is a solution to the system in Equation~\eqref{eq:system for rational eigenfunctions} must also satisfy $g(P)=0$ for all $g\in I$. 
The ideal $I$ is  the infinite set of \emph{all polynomial consequences} of the system defined by the $f_i$'s, meaning that \emph{any} polynomial relation on any (subset of) indeterminates of $R$ which is implied by \eqref{eq:system for rational eigenfunctions} lies in $I$. 
This is key to answering the question:  if there is a polynomial relation that holds among the coefficients $a_i$'s and $b_i$'s, then it can be recovered from $I$ by a suitable symbolic computation. Let us discover what that computation entails. 

Define $V(I)$ to be the set of all solutions to the system $f_1=\dots=f_{10}=0$; it is equivalent to (cf.\ Chapter 4 and Proposition 4 in \S 1.4. of \cite{CLO}): 
\begin{equation}\label{eq:variety of the system for rational eigenfunctions}
    V(I) = \{P\in\mathbb C^{17}: g(P)=0 \mbox{ for all } g\in I\}.
\end{equation} 
This set is called the affine algebraic \emph{variety} of the system, and it lives in the $17$-dimensional affine space $\mathbb C^{17}$. (In general, an algebraic variety is defined simply a set of common solutions to a system of polynomial equations.) 

Next, we know that the system~\ref{eq:system for rational eigenfunctions} always has a solution, namely the trivial one. Therefore we remove the trivial solutions where  $\varphi$, $\lambda$, or the denominator of $\varphi$ is zero. 
Let $P=(P_1,\dots,P_{17})\in\mathbb C^{17}$ be an arbitrary solution. The set of nontrivial solutions is therefore the following:  
\[ 
V(I)\setminus\{P:P_1=0\}\setminus \{P:P_2=P_3=P_4=0\} \setminus \{ P:P_5=P_6=P_7=0\}. 
\]
The key observation is that the sets of points that are `removed' from the variety $V(I)$ are each varieties, as they are defined by polynomial equations setting particular coordinates to zero.  

\begin{proposition}
    Let $V(I)$ be the variety defined by the system~\ref{eq:system for rational eigenfunctions}, as defined in \ref{eq:variety of the system for rational eigenfunctions}. 
    The coefficients of the ODE ($a_i$ and $b_i$) that will result in a nontrivial solution $(\lambda,c_0,c_1,c_2,d_0,d_1,d_2)$ to Eq.~\ref{eq:system for rational eigenfunctions} (values for $c_i, d_i, \lambda$) lie in the following algebraic variety: 
    \begin{equation}\label{eq:the rational eigenfunction variety}
        \mathcal E := W\cap \mathbb C^{10}, \mbox{ where } W = V(I) \setminus V(\lambda)\setminus V(c_0,c_1,c_2)\setminus V(d_0,d_1,d_2).
    \end{equation}
    The intersection with $\mathbb C^{10}$ is on the last $10$ coordinates of $\mathbb C^{17}$ corresponding to the parameters $a_1,\dots,a_5,b_1,\dots,b_5$. 
    The set of restrictions that the coeffcieints of the ODE must satisfy in order for the nontrivial solution to exist is therefore the following ideal: 
    \begin{equation}\label{eq:elimination ideal}
        I(\mathcal E) = I(W)\cap \mathbb Q[a_1,\dots,a_5,b_1,\dots,b_5],
    \end{equation}
    where 
    \begin{equation}\label{eq:removing components by saturation}
        I(W) = (( I(V) : \left<\lambda\right>^\infty ) :\left<c_0,c_1,c_2\right>^\infty ) :\left<d_0,d_1,d_2\right>^\infty. 
    \end{equation}
\end{proposition}

\begin{definition}\label{defn:the rational eigenfunction variety}
We will name $\mathcal E$ from \eqref{eq:the rational eigenfunction variety} the  \emph{rational eigenfunction variety} for the ODE. It is a complex affine algebraic variety. Its structure is in next subsection.
\end{definition}

\begin{remark}
    $\mathcal E$ is a \emph{complex} algebraic variety. Real parts of algebraic varieties are not easy to obtain; but numerical algebraic geometry software can help with that aspect, if needed. 
    
    Also, we typically work over algebraically closed fields, so coefficients of the system are in $\mathbb C$. However, when we carry out symbolic computations, we seek exact symbolic answers, we work over an exact field such as $\mathbb Q$. This is just the standard way things are done in computational algebraic geometry (and even for applications). 
\end{remark}
\begin{proof}[Proof of Proposition]
    The statements in \eqref{eq:the rational eigenfunction variety} is merely a translation of the  equation above it. Namely, the set $\{P:P_1=0\}$ corresponds to the set of points in $\mathbb C^{17}$ for which we require $\lambda=0$, and so the corresponding variety is the variety of the polynomial $\lambda$, and so on. $W$ is obtained from $V(I)$ by removing trivial solutions; the intersection with the $10$-dimensional space then restricts the variety to a subvariety on the relevant coordinates, $a_i$'s and $b_i$'s. 
    
    For two ideals $I$ and $J$, the operation $I:J^\infty$ is called \emph{saturation}.     For reference, ideal saturations are covered in Section 4 of Chapter 4 of \cite{CLO}. 
    The variety $V(I:J^\infty)$ is the Zariski closure of $V(I)\setminus V(J)$;  without delving into all of the details we note that this is the smallest variety containing the set in which we are interested. 
    Saturation is the algebraic operation on ideals which removes corresponding subvariety from the larger variety. Thus, Equation~\eqref{eq:removing components by saturation} is also a restatement of the algebraic geometry facts. 
    
    The final step is to eliminate unwanted indeterminates. The system defining $I$ was originally considered in $17$ unknowns; but once we have removed trivial solutions, we are interested in the polynomial consequences of $I$ that involve the indeterminates $a_1,\dots,a_5,b_1,\dots,b_5$ only. Algebraically, this means we need to compute the \emph{elimination ideal} $I(\mathcal E) = I(W)\cap \mathbb Q[a_1,\dots,a_5,b_1,\dots,b_5]$; this is Equation~\eqref{eq:elimination ideal}. 
    This ideal can be computed by Gr\"obner bases; see Chapter 3 of \cite{CLO}. 
\end{proof}

\subsection{The geometric consequence of the algebraic computation}

We computed the ideal $I(\mathcal E)$ using the computer algebra system {\tt Macaulay2} \cite{M2}. 
It is instructive to see the intermediate computational outputs:
The first saturation computation, $I(V) : \left<\lambda\right>^\infty$, returns $228$ polynomials of degrees between $2$ and $7$.  
The second saturation (removing the points where all $c_i=0$) returns $195$ ideal generators of degrees $2$ to $5$.  
The last saturation (removing points where all $d_i=0$) returns $404$ polynomials of degree $2$ to $6$. 
Finally, elimination to preserve only the polynomials in $a_i$'s and $b_i$'s returns $10$ polynomials all of degree $4$. 
These $10$ quartics are therefore \emph{the defining equations of the complex rational eigenfunction variety} for the ODE. 
The quartics can be found listed on the \href{https://github.com/mmtree/linear_rational_eigenfunctions}{GitHub repository}.

Now, the question is: what \emph{is} the set of $a_i$'s and $b_i$'s that satisfy these ten quartic equations? We ask {\tt Macaulay2} for some essential information, such as dimension and degree and the decomposition into irreducible components. The following is the summary: 
\begin{equation}
    \mathcal E = \mathcal Q \cup \mathcal L,
\end{equation}
where $\mathcal L$ is a $6$-dimensional \emph{linear} space defined by the $4$ linear polynomials\footnote{Meaning, the constraints are obtaining by  setting these linear polynomials to $=0$.} 
\[ b_3,\quad a_5,\quad a_4-b_5,\quad a_3-b_4,
\] 
and $\mathcal Q$ is an irreducible variety in $\mathbb C^{10}$ of dimension $8$ cut out by 
$1$ cubic and $6$ quartics.  The $7$ polynomials defining $\mathcal E$ are easy to obtain using the {\tt Macaulay2} code provided.

The following is thus a complete description of the set of complex points $(a_1..a_5,b_1..b_5)$ for which the system \eqref{eq:system for rational eigenfunctions} has a nontrivial solution. 
\begin{corollary}
    The ODE has a nontrivial rational eigenfunction if and only if 
    
    $a_1,\dots,a_5,b_1,\dots,b_5\in\mathbb C^{10}$ satisfy: 
    \begin{itemize}
        \item either all four of the equations defining the linear variety $\mathcal L$; 
        \item or all seven ($1$ cubic and $6$ quartics) equations defining the $8$-dimensional  variety $\mathcal Q$. 
    \end{itemize}

    In particular, with probability $1$, a randomly generated point in $\mathbb C^{10}$ will \emph{not} satisfy these criteria. 
\end{corollary}
The last statement holds because the algebraic variety $\mathcal E$ is $8$-dimensional and therefore a set of measure zero in the 10-dimensional space.  

Next, we wish to explore the subspace $\mathcal Q$, as it appears that the \emph{points in its singular locus}  give rise to a pair of linearly independent solutions. By definition of singularities, these points are such that $Jacobian(f_1,\dots,f_7)$ becomes the zero matrix; namely, $codim(\mathcal Q)=2$, therefore the singular locus are the points for which the Jacobian's $1\times1$ minors vanish. 

Macaulay2 computation reveals that all points on $\mathcal X$ that satisfy the condition  $$Jacobian(f_1,\dots,f_7)=0$$ satisfy the following 9 quadratic equations: 
\begin{align*}
     4a_5b_3-a_4b_4,            \\
 2a_4b_3-2a_3b_4+b_4^2-4b_3b_5,  \\
      2a_2b_3-a_1b_4+b_2b_4-2b_1b_5,  \\
      2a_5b_1-a_2b_4,                 \\
       a_4b_1-a_1b_4+b_2b_4-2b_1b_5,   \\
      2a_3b_1-2a_1b_3+2b_2b_3-b_1b_4, \\
      a_4^2-4a_3a_5+2a_5b_4-2a_4b_5,  \\
      a_2a_4-2a_1a_5+2a_5b_2-2a_2b_5, \\
      2a_2a_3-a_1a_4+a_4b_2-a_2b_4   
\end{align*}
The set of points satisfying these equations is 6-dimensional in the 10-dimensional ambient space. 
These are precisely the equations defining the variety $\mathcal X$, from Section~\ref{sec:solvable_ODEs_eig_params}.

\section{Alternative approach to obtaining ODE and eigenfunction parameters}\label{sec:alt_method}

In an alternative approach, we initially do not assume anything about the form of $\frac{dx}{dt}$ and $\frac{dy}{dt}$ (we do not impose a quadratic polynomial form), but instead initially impose a structure on the eigenfunctions.  We solve for $\frac{dx}{dt}$ and $\frac{dy}{dt}$
as functions of two independent eigenfunctions $\varphi_1$ and $\varphi_2$ as well as their partial derivatives and eigenvalues, and then impose restrictions on the output to obtain quadratic forms for $\frac{dx}{dt}$ and $\frac{dy}{dt}$.

We begin by applying the eigenfunction equation to $\varphi_1$ and $\varphi_2$,
\begin{align*}
        \frac{\partial \varphi_1}{\partial x} \frac{dx}{dt} + \frac{\partial \varphi_1}{\partial y} \frac{dy}{dt} &= \lambda_1 \varphi_1, \\
         \frac{\partial \varphi_2}{\partial x} \frac{dx}{dt} + \frac{\partial \varphi_2}{\partial y} \frac{dy}{dt} &= \lambda_2 \varphi_2. 
\end{align*}

We can write this system of equations using a matrix formulation,
\begin{align*} 
       \begin{bmatrix}
        \frac{\partial \varphi_1}{\partial x} & \frac{\partial \varphi_1}{\partial y} \\
         \frac{\partial \varphi_2}{\partial x} & \frac{\partial \varphi_2}{\partial y}
       \end{bmatrix}
       \begin{bmatrix}
        \frac{dx}{dt} \\
        \frac{dy}{dt}
       \end{bmatrix}
       = 
        \begin{bmatrix}
        \lambda_1 \varphi_1 \\
        \lambda_2 \varphi_2
       \end{bmatrix}.
\end{align*}
This system of equations has a unique solution if the determinant of the matrix is nonzero,

$$\frac{\partial \varphi_1}{\partial x} \frac{\partial \varphi_2}{\partial y} - \frac{\partial \varphi_2}{\partial x} \frac{\partial \varphi_1}{\partial y}  \neq 0.$$

Solving this system for $\frac{dx}{dt}$ and $\frac{dy}{dt}$ produces
\begin{align}\label{eq:appx_odes_gen1}
\begin{split}
        \frac{dx}{dt}  &= \frac{\frac{\partial \varphi_2}{\partial y} \lambda_1 \varphi_1 - \frac{\partial \varphi_1}{\partial y} \lambda_2 \varphi_2}{  \frac{\partial \varphi_1}{\partial x} \frac{\partial \varphi_2}{\partial y} - \frac{\partial \varphi_2}{\partial x} \frac{\partial \varphi_1}{\partial y}},\\
    \frac{dy}{dt}  &= \frac{ \frac{\partial \varphi_1}{\partial x} \lambda_2 \varphi_2 - \frac{\partial \varphi_2}{\partial x} \lambda_1 \varphi_1 }{ \frac{\partial \varphi_1}{\partial x}  \frac{\partial \varphi_2}{\partial y} -\frac{\partial \varphi_2}{\partial x} \frac{\partial \varphi_1}{\partial y} }. 
    \end{split}
\end{align}

We now impose a linear rational structure for the eigenfunctions,
\begin{align*}
    \varphi_1(x,y) &= \frac{c_0 + c_1 x + c_2 y}{d_0 + d_1 x + d_2 y}, \\
    \varphi_2(x,y) &= \frac{k_0 + k_1 x + k_2 y}{m_0 + m_1 x + m_2 y}.
\end{align*}
Normalizing the eigenfunctions by setting $c_1 = d_1 = k_1 = m_1 = 1$, and further adding the restriction that $d_0 = m_0 = 0$, as our previous results indicate that there exists a set of eigenfunction solutions that lack a constant in either the numerator or denominator, we further restrict our eigenfunctions to the form,
\begin{align*}
    \varphi_1(x,y) &= \frac{c_0 + x + c_2 y}{ x + d_2 y}, \\
    \varphi_2(x,y) &= \frac{k_0 + x + k_2 y}{ x + m_2 y}.
\end{align*}
Using this form for the eigenfunctions in Eqs.~\ref{eq:appx_odes_gen1} produces
\begin{align*}
\frac{dx}{dt}  &= \frac{N_1(x,y)}{D_1(x,y)}, \\
\frac{dy}{dt}  &= \frac{N_2(x,y)}{D_2(x,y)}, \\
\end{align*}
where the numerators are
\begin{align*}
\begin{split}
N_1(x,y) &= (c_0 d_2 k_0 m_2(\lambda_1 - \lambda_2)) y \\
&+ (d_2 m_2 (c_2 k_0 \lambda_1 - c_0 k_2 \lambda_2))y^2 \\
&+ (c_0 k_0 (m_2 \lambda_1 - d_2 \lambda_2))x \\
&+ (k_0 m_2(d_2(\lambda_1 - \lambda_2) + c_2 (\lambda_1 + \lambda_2)) - c_0 d_2(m_2(\lambda_2 - \lambda_1) + k_2(\lambda_1 + \lambda_2))) xy \\
&+ (c_2 d_2 (m_2 - k_2) \lambda_1 + c_2 k_2 m_2 \lambda_2 - d_2 k_2 m_2 \lambda_2)x y^2 \\
&+ (\lambda_1 (m_2 - k_2) + \lambda_2 (c_2 - d_2))x^2,
\end{split}\\
\begin{split}
    N_2(x,y) &= (c_0 k_0 (d_2 \lambda_1 - m_2 \lambda_2)) y \\
    &+ (d_2 k_0 m_2 \lambda_2 + c_2 k_0 (d_2 \lambda_1 - m_2 \lambda_2) + c_0(d_2 k_2 \lambda_1 - d_2 m_2 \lambda_1 - k_2 m_2 \lambda_2)) y^2 \\
    &+ (c_2 d_2 (k_2 - m_2)\lambda_1 - c_2 k_2 m_2 \lambda_2 + d_2 k_2 m_2 \lambda_2) y^3 \\
    &+ (c_0 k_0 (\lambda_1 - \lambda_2)) x \\
    &+ (c_2 k_0(\lambda_1 - \lambda_2) + c_0 k_2(\lambda_1 - \lambda_2) + d_2 k_0 (\lambda_1 + \lambda_2) - c_0 m_2 (\lambda_1 + \lambda_2)) xy \\
    &+ (c_2 k_2 (\lambda_1 - \lambda_2) + d_2 m_2 (\lambda_2 - \lambda_1) + d_2 k_2 (\lambda_1 + \lambda_2) - c_2 m_2 (\lambda_1 + \lambda_2)) xy^2 \\
    &+ (k_0 \lambda_1 - c_0 \lambda_2) x^2 \\
    &+ (k_2 \lambda_1 - m_2 \lambda_1 + (d_2 - c_2)\lambda_2) x^2 y,
\end{split}
\end{align*}
and the denominators are
\begin{align*}
\begin{split}
D_1(x,y) &=c_0 k_0 (d_2 - m_2) \\
&+ (k_0 (d_2 - c_2) + c_0 (k_2 - m_2)) x \\
&+ (c_0 d_2 (k_2 - m_2) + k_0 m_2(d_2 - c_2)) y, \\
\end{split}\\
\begin{split}
D_2(x,y) &=c_0 k_0 (m_2 - d_2)\\
&+ (k_0 (c_2 - d_2) + c_0(m_2 - k_2))x\\
&+ (k_0 m_2(c_2 - d_2) + c_0 d_2 (m_2 - k_2))y. \\
\end{split}
\end{align*}

We wish to constrain the terms so that $\frac{dx}{dt}$ and $\frac{dy}{dt}$ are both quadratic polynomials.  This means that both denominators should be nonzero constants --- all terms in the denominator should be zero other than the $x^0 y^0$ term.  The numerators should not contain any terms higher than second order.  Imposing these constraints results in a polynomial system of equations, similar to how the system Eq.~\ref{eq:system for rational eigenfunctions} was derived by setting terms in the previous equation equal to zero. 
In this way, we can solve for eigenfunctions that produce quadratic ODEs and determine the family of quadratic ODEs that have independent eigenfunctions of the linear rational form that we posed.

\section*{Data and Code Availability Statement}
The defining equations of the complex rational eigenfunction variety and the code used to generate the results in this paper are available on GitHub at \url{https://github.com/mmtree/linear_rational_eigenfunctions}.

\section*{Acknowledgments}
The authors are grateful to Alasdair Hastewell for insightful conversations about using rational eigenfunctions to solve ODEs.
The authors are also grateful to Shaowei Lin for an interesting discussion about the singularities of the rational eigenfunction variety and how they should be studied in general.

\bibliographystyle{plain}


\end{document}